\documentclass{amsart}

\usepackage{graphicx}
\usepackage{amsfonts}
\usepackage{textcomp}
\usepackage{url}
\usepackage{amsmath}
\usepackage{enumitem}
\usepackage{hyperref}
\usepackage{xcolor}
\usepackage{graphicx}
\hypersetup{
	colorlinks=true,
	linkcolor=blue,
	filecolor=magenta,      
	urlcolor=cyan,
}

\newtheorem{theorem}{Theorem}[section]
\newtheorem*{theorem*}{Theorem}
\newtheorem{lemma}[theorem]{Lemma}

\newtheorem{proposition}[theorem]{Proposition}

\theoremstyle{definition}
\newtheorem{definition}[theorem]{Definition}

\theoremstyle{remark}
\newtheorem{remark}[theorem]{Remark}

\numberwithin{equation}{section}

\newcommand{\union}{\mathop{\bigcup}}

\newcommand{\re}{\mathrm{Re}}
\newcommand{\im}{\mathrm{Im}}

\newcommand{\Int}{\mathrm{int}}

\newcommand{\del}{\partial}
\newcommand{\norm}[1]{\left\lVert#1\right\rVert}
\newcommand{\htop}{h_{\mathrm{top}}}

\newcommand{\SL}{\mathrm{SL}}
\newcommand{\cl}{\mathrm{cl}}

\newcommand{\N}{\mathbb N}
\newcommand{\Z}{\mathbb Z}

\newcommand{\R}{\mathbb R}
\newcommand{\C}{\mathbb C}
\newcommand{\T}{\mathbb T}

\newcommand{\As}{\mathcal A}
\newcommand{\Ss}{\mathcal S}

\newcommand{\Bs}{\mathcal B}
\newcommand{\Ps}{\mathcal P}

\newcommand{\Ns}{\mathcal N}

\newcommand{\Sbb}{\mathbb S}

\newcommand{\diver}{\mathrm{div}}

\newcommand{\Fs}{\mathcal F}
\newcommand{\Us}{\mathcal U}
\newcommand{\Hs}{\mathcal H}

\let\oldphi\phi
\let\phi\varphi

\let\epsilon\varepsilon
\let\oldbar\bar
\let\bar\overline

\begin{document}

\title{Polynomial decay of correlations of pseudo-Anosov diffeomorphisms}


\author{Dominic Veconi}
\address{Department of Mathematics, Wake Forest University, Winston-Salem, NC, USA.}
\email{veconid@wfu.edu}



\maketitle

\begin{abstract}
	We give a construction of a smooth realization of a pseudo-Anosov diffeomorphism of a Riemannian surface, and show that it admits a unique SRB measure with polynomial decay of correlations, large deviations, and the central limit theorem. The construction begins with a linear pseudo-Anosov diffeomorphism whose singularities are fixed points. Near the singularities, the trajectories are slowed down, and then the map is conjugated with a homeomorphism that pushes mass away from the origin. The resulting map is a $C^{2+\epsilon}$ diffeomorphism topologically conjugate to the original pseudo-Anosov map. To prove that this map has polynomial decay of correlations, our main technique is to use the fact that this map has a Young tower, and study the decay of the tail of the first return time to the base of the tower. 
\end{abstract}


\section{Introduction}

In \cite{Kat79}, A. Katok introduced a $C^\infty$ area-preserving diffeomorphism of $\T^2$ with \emph{strict non-uniform hyperbolicity:} the map $G_{\T^2} : \T^2 \to \T^2$ (known as the \emph{Katok map of the torus}) has nonzero Lyapunov exponents Lebesgue-almost everywhere, but admits a singularity at the origin at which the differential $dG_{\T^2}$ is the identity. As a consequence, there are trajectories admitting zero Lyapunov exponents, and so the Lyapunov exponents of $G_{\T^2}$ come arbitrarily close to 0 over $\T^2$ (this is the sense in which the nonuniform hyperbolicity is strict). The map $G_{\T^2}$ is a Bernoulli automorphism (meaning $G_{\T^2}$ is measurably isomorphic to a Bernoulli shift), and is topologically conjugate to a linear Anosov diffeomorphism of $\T^2$. Then in \cite{GerbKatPA}, A. Katok and M. Gerber extended the construction of the Katok map to any compact Riemannian surface, presenting a wide family of area-preserving $C^\infty$ diffeomorphisms that are both strictly non-uniformly hyperbolic and Bernoulli. However, unlike the Katok map $G_{\T^2}$, the Bernoulli diffeomorphisms in \cite{GerbKatPA} are not topologically conjugate to an Anosov map. Indeed, Anosov diffeomorphisms on surfaces admit global stable and unstable 1-dimensional foliations, and no surface with genus $\neq 1$ admits 1-dimensional foliations, so the only surface that admits Anosov diffeomorphisms is the torus $\T^2$. So instead of being conjugate to an Anosov diffeomorphism, the non-uniformly hyperbolic diffeomorphisms in \cite{GerbKatPA} are topologically conjugate to the broader class of \emph{pseudo-Anosov diffeomorphisms.} 

Pseudo-Anosov diffeomorphisms were introduced by W. Thurston in \cite{BThursty} as a generalization of Anosov diffeomorphisms: rather than admitting global stable and unstable submanifolds, pseudo-Anosov maps admit stable and unstable \emph{foliations with singularities}, for which there are a finite number of singularities where multiple leaves of the foliations meet (see Section \ref{subsec:PAHs}). In the theory of mapping class groups, pseudo-Anosov diffeomorphisms play a role in the \emph{Nielsen-Thurston classification} of surface homeomorphisms: 

\begin{theorem*}
	Let $M$ be a compact orientable surface, and let $f : M \to M$ be a homeomorphism. Then $f$ is isotopic to a homeomorphism $F : M \to M$ satisfying exactly one of the following three conditions: 
	\begin{itemize}
		\item $F$ is a rotation: there is an integer $n \geq 1$ for which $F^n = \mathrm{id}$. 
		\item $F$ is a Dehn twist: there is a closed curve that $F$ leaves invariant. 
		\item $F$ is pseudo-Anosov. 
	\end{itemize}
\end{theorem*} 

The construction of the Katok map $G_{\T^2}$ and the smooth realizations of the pseudo-Anosov maps in \cite{GerbKatPA} both use a similar \emph{slow-down procedure.} In the case of the Katok map $G_{\T^2}$, the construction starts with a linear hyperbolic toral automorphism $f : \T^2 \to \T^2$ induced by a matrix $A \in \SL(2,\Z)$. The map $f$ is then written in coordinates near the origin as the time-1 map of the flow induced by the system of ODEs
\[
\dot s_1 = s_1 \log\lambda, \quad \dot s_2 = -s_2 \log\lambda 
\]
(where $(s_1, s_2)$ are the coordinates induced by the eigenvectors of $A$ with eigenvalues $\lambda > 1$ and $\lambda^{-1} < 1$, respectively). The trajectories of the flow near the origin are then ``slowed down,'' so that the time-1 map of the resulting flow has a differential at the origin equal to the identity. Finally, the new time-1 map is conjugated with a homeomorphism that makes the resulting map area-preserving; the final map is $G_{\T^2}$. 

In constructing the non-uniformly hyperbolic surface diffeomorphisms in \cite{GerbKatPA}, Gerber and Katok begin with a pseudo-Anosov homeomorphism $f : M \to M$ whose singularities are fixed points, construct a continuous vector field in coordinates around each singularity whose time-1 map is $f$, and similarly slow down the vector field trajectories to produce a new time-1 diffeomorphism $g : M \to M$ whose differential at the singularities is the identity. Importantly, the initial pseudo-Anosov map $f$ is not a true diffeomorphism: the map is necessarily not smooth at the singularities of the foliation for $f$. The slowdown procedures used to construct the pseudo-Anosov maps in \cite{GerbKatPA} and the Katok map $G_{\T^2}$ in \cite{Kat79} are similar, but the slowdown procedure used to construct $G_{\T^2}$ in \cite{Kat79} is presented very generally, giving great flexibility with the rate at which the flow trajectories slow down. In contrast, for the pseudo-Anosov maps in \cite{GerbKatPA}, a specific slowdown rate is given. The resulting slowed-down diffeomorphism preserves the area given by the coordinates around each singularity, so no conjugating map is required to further make the diffeomorphism area-preserving. 

In \cite{PSS}, Y. Pesin, S. Senti, and F. Shahidi showed that the Katok map $G_{\T^2}$ has a range of thermodynamic properties: they demonstrate that the unique SRB measure of $G_{\T^2}$ has polynomial decay of correlations (rate of mixing), the central limit theorem, and polynomial large deviations with respect to H\"older observables. They also demonstrate (also in \cite{PSZ17}) that $G_{\T^2}$ has a unique measure of maximal entropy, with respect to which $G_{\T^2}$ has exponential decay of correlations and the central limit theorem. Furthermore, as their main result, they give a construction of a non-uniformly hyperbolic diffeomorphism of any compact surface that comes from gluing the singularity of $G_{\T^2}$ to the surface, and the resulting diffeomorphism has the same thermodynamic properties as the Katok map. 

The goal of this paper is to produce a smooth realization of a pseudo-Anosov diffeomorphism that enjoys these same properties. We provide an alternative construction to prove the main result in \cite{PSS} that any surface has a non-uniformly hyperbolic diffeomorphism with the described thermodynamic properties. Unlike the construction in \cite{PSS}, which begins with the Katok map and uses a sequence of maps $\T^2 \to \Sbb^2 \to D^2 \to M$  to produce a semi-conjugacy between $G_{\T^2}$ and a map of $M$ (where $\Sbb^2$ and $D^2$ are the 2-sphere and the 2-disk, respectively), our construction produces a diffeomorphism that is topologically conjugate to a homeomorphism $f : M \to M$ that is \emph{a priori} independent of maps on other surfaces. 

The techniques in \cite{PSS} are based on modeling the Katok map with a \emph{Young tower}, a symbolic representation of hyperbolic maps by a tower whose base is conjugate to a countable-state Bernoulli shift, originally introduced in \cite{Young}. In \cite{VecPAD1}, we used Young towers to show that the smooth realizations of pseudo-Anosov maps in \cite{GerbKatPA} have a unique measure of maximal entropy with exponential decay of correlations and the Central Limit Theorem with respect to H\"older potentials; and furthermore, the geometric $t$-potentials $\phi_t(x) = -t\log\left|dg|_{E^u(x)}\right|$ admit unique equilibrium states for $t \in (t_0, 1)$ for some $t_0 < 0$ with exponential decay of correlations and the central limit theorem, while the geometric potential $\phi_1(x) = -\log\left|dg|_{E^u(x)}\right|$ has two classes of equilibrium states: a unique SRB measure, and convex combinations of Dirac masses at the singularities. These results mirror the results on the Katok map presented in \cite{PSZ17}. In both \cite{PSZ17} and \cite{VecPAD1}, proving that the Katok map and the pseudo-Anosov smooth realizations admit Young towers required careful examination of the behavior of the trajectories near the neutral fixed point singularities. Additionally, it was necessary to show that the number of partition elements of the Young tower with a given inducing time (first-return time in this case) is exponentially bounded with an exponent strictly less than the topological entropy. Both of these technical challenges could be handled similarly for the systems discussed in both \cite{PSZ17} and \cite{VecPAD1}. However, the arguments proving polynomial decay of correlations for the Katok map in \cite{PSS} require that the slow-down exponent $\alpha > 0$ of the trajectories satisfy $\alpha < \frac 1 3$; the slow-down rate of the pseudo-Anosov smooth realizations in \cite{GerbKatPA}, on the other hand, is specifically chosen to be $\alpha = \frac{p-2}{p}$, where $p \geq 3$ is the number of \emph{prongs of the singularity} (see Section \ref{subsec:PAHs}). When $p \geq 4$, this slow-down exponent falls outside of the range for which the arguments in \cite{PSS} can be directly applied. One of the goals of this paper is to adapt the construction of the smooth realization of pseudo-Anosov maps in \cite{GerbKatPA} in a way that provides more flexibility for producing a non-uniformly hyperbolic surface diffeomorphism with different topological and ergodic properties. In particular, we produce a non-uniformly hyperbolic $C^{2+\epsilon}$ Bernoulli diffeomorphism that is topologically conjugate to a pseudo-Anosov homeomorphism, and whose unique SRB measure has polynomial decay of correlations, the central limit theorem, and polynomial large deviations. 

We remark that many examples of strictly non-uniformly hyperbolic dynamical systems are constructed by inducing a \emph{neutral fixed pont}, which is a fixed point whose differential is the identity (as is done for the Katok map and for the pseudo-Anosov smooth realizations). In many of these cases, the resulting map also has a unique SRB measure with polynomial decay of correlations. In the one-dimensional category, the \textit{Manneville-Pomeau} map $f : I \to I$ is a non-uniformly expanding map with a fixed point at which the derivative is 1 \cite{MP}; it has been shown that the Manneville-Pomeau map and other related one-dimensional transformations have a unique invariant measure absolutely continuous to Lebesgue, with respect to which the map admits polynomial decay of correlations \cite{GouDecay, HuDecay, LSV}. Additionally, L.S.-Young introduced in \cite{Young-circle} an expanding homeomorphism of the circle with an indiffierent fixed point at the origin that has polynomial decay of correlations for its unique SRB measure. On surfaces, in addition to the examples of diffeomorphis with indifferent fixed points considered in \cite{PSS}, H. Hu constructed a different class of diffeomorphisms with indifferent fixed points called ``almost Anosov'' maps \cite{HuAA}, and demonstrated with X. Zhang in \cite{HuZhang} that these diffeomorphisms also have polynomial decay of correlations. In the category of dissipative diffeomorphisms with hyperbolic attractors, J. Alves and V. Pinheiro constructed in \cite{AP} a nonuniformly hyperbolic solenoid map on $\T^3$ with an indifferent fixed point. They showed that this map admits a Young structure, and used this Young structure to show that the ``solenoid with intermittency'' has polynomial decay of correlations. Finally, S. Burgos in \cite{Burgos} considered a dynamical system with a uniformly hyperbolic attractor, which has a hyperbolic fixed point that can be slowed down to an indifferent fixed point (following the procedure in \cite{CDP}). Using techniques from \cite{PSS} and \cite{Z}, Burgos showed that this map's unique SRB measure has polynomial decay of correlations. This dissipative map studied in \cite{Burgos, CDP, Z} is another example of a strictly nonuniformly hyperbolic diffeomorphism whose indifferent fixed point is induced from the \emph{slow-down procedure}, the procedure introduced in \cite{Kat79} and used in \cite{PSZ17, PSS} to study area-preserving diffeomorphisms. The pseudo-Anosov smooth realizations in \cite{GerbKatPA, VecPAD1} also use a similar slow-down procedure; in this paper, we generalize the procedure from \cite{GerbKatPA} to produce a family of diffeomorphisms with a wide range of ergodic properties, including polynomial decay of correlations for the unique SRB measure. 


 This paper is organized as follows. In Section 2, we introduce the preliminary definitions needed for our main results; in particular, pseudo-Anosov homeomorphisms and the relevant statistical properties. In Section 3, we state our main results. The construction of the diffeomorphism $g$ is given in Section 4, and we show in Section 5 that the resulting map has a Young tower. In Section 6, we study different technical estimates near the singularities of $g$. The tail of the return time is estimated in Sections 7 and 8, and in Section 9, we prove the main result. 

\section{Preliminaries}

\subsection{Pseudo-Anosov maps}\label{subsec:PAHs}

Before we define pseudo-Anosov homeomorphisms and construct their smooth realizations, we briefly discuss measured foliations with singularities. Our exposition is adapted from the presentation in \cite{BaPeNUH}, Section 6.4. For the reader's convenience, we have restated their exposition here. Also see Section 2 of \cite{VecPAD1}. 

\begin{definition}
	A \emph{measured foliation with singularities} is a triple $(\Fs, S, \nu)$, where: 
	\begin{itemize}
		\item $S = \{x_1, \ldots, x_m\}$ is a finite set of points in $M$, called \emph{singularities}; 
		\item $\Fs = \tilde \Fs \uplus \Ss$ is a partition of $M$, where $\Ss$ is a partition of $S$ into points and $\tilde \Fs$ is a smooth foliation of $M \setminus S$;
		\item $\nu$ is a \emph{transverse measure}; in other words, $\nu$ is a measure defined on each curve on $M$ transverse to the leaves of $\tilde \Fs$;
	\end{itemize}
	and the triple satisfies the following properties:
	\begin{enumerate} 
		\item There is a finite atlas of $C^\infty$ charts $\oldphi_k : U_k \to \C$ for $k = 1, \ldots, \ell$, $\ell \geq m$. 
		\item For each $k = 1, \ldots, m$, there is a number $p = p(k) \geq 3$ of elements of $\tilde\Fs$ meeting at $x_k \in S$ (these elements are called \emph{prongs} of $x_k$) such that: 
		\begin{enumerate}[label=(\alph*)]
			\item $\oldphi_k(x_k) = 0$ and $\oldphi_k(U_k) = D_{a_k} := \{z \in \C : |z| \leq a_k\}$ for some $a_k > 0$; 
			\item if $C \in \tilde\Fs$, then the components of $C \cap U_k$ are mapped by $\oldphi_k$ to sets of the form 
			\[
			\left\{z \in \C: \im\left(z^{p/2}\right) = \mathrm{constant} \right\} \cap \oldphi_k(U_k);
			\]
			\item the measure $\nu|U_k$ is the pullback under $\oldphi_k$ of $$\left| \im\left(dz^{p/2}\right)\right| = \left| \im\left(z^{(p-2)/2} dz\right)\right|.$$
		\end{enumerate}
		\item For each $k > m$, we have: 
		\begin{enumerate}[label=(\alph*)]
			\item $\oldphi_k(U_k) = (0, b_k) \times (0,c_k) \subset \R^2 \approx \C$ for some $b_k, c_k > 0$; 
			\item If $C \in \tilde\Fs$, then components of $C \cap U_k$ are mapped by $\oldphi_k$ to lines of the form
			\[
			\{z \in \C : \im \,z = \mathrm{constant}\} \cap \oldphi_k(U_k).
			\] 
			\item The measure $\nu|U_k$ is given by the pullback of $|\im \,dz|$ under $\oldphi_k$. 
		\end{enumerate}
	\end{enumerate}
\end{definition}

A singularity with $p=3$ prongs is shown in Figure 1. 


\begin{remark}
	Henceforth, we refer to the $C^\infty$ curves that are elements of $\Fs$ as ``leaves (of the foliation)''; in particular, despite the technical fact that the singleton sets of singularities $\{x_1\}, \ldots, \{x_k\}$ are elements of $\Fs$, we do not refer to these points when we refer to ``leaves of the foliation''. 
\end{remark}

\begin{figure}
	\centering
	\includegraphics[width=0.6\textwidth]{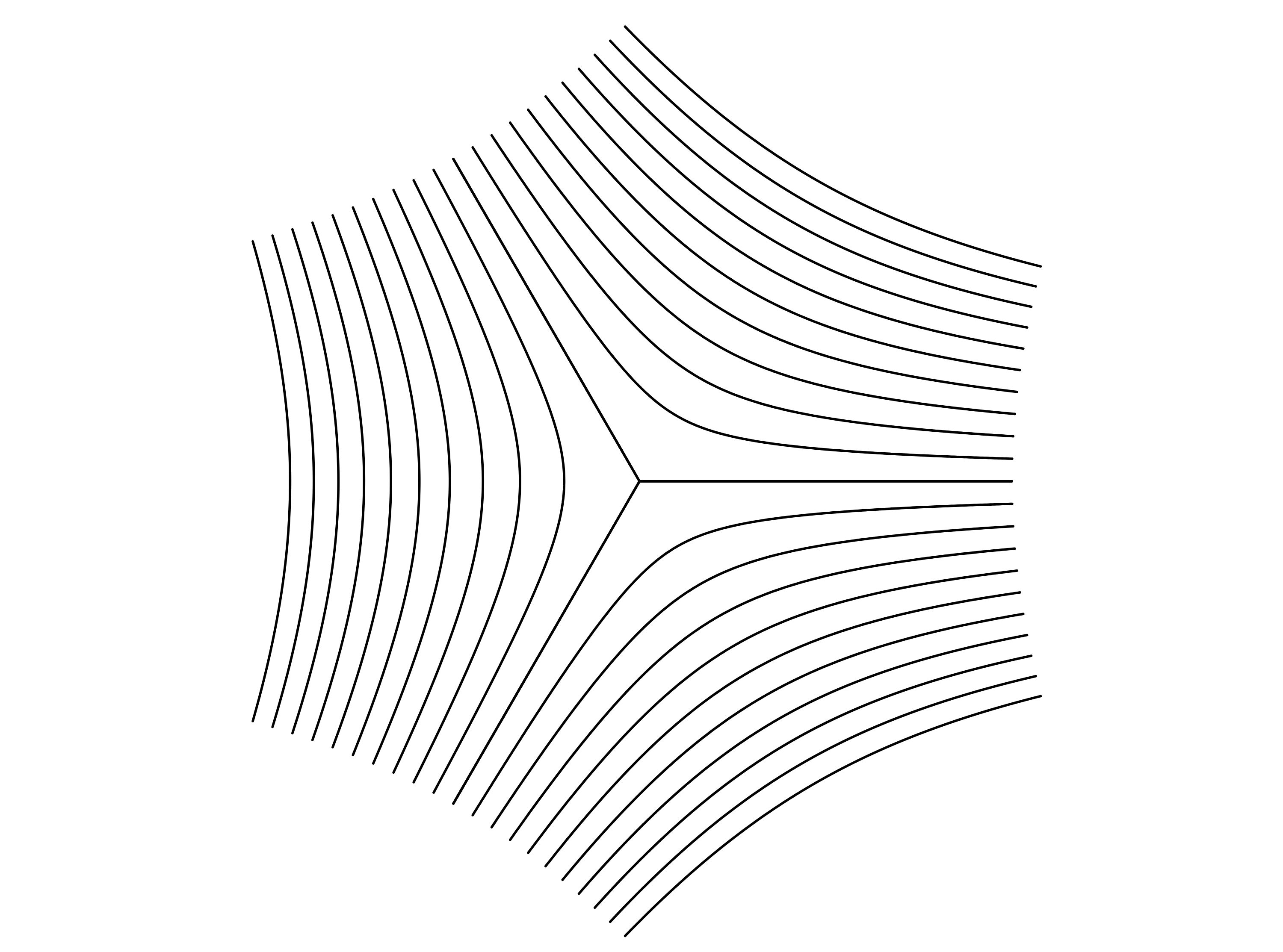}
	\caption{A 3-pronged singularity of a measured foliation with singularities.}
\end{figure}

\begin{definition}\label{PAH-def}
	A surface homeomorphism $f$ of a manifold $M$ is \emph{pseudo-Anosov} if there are measured foliations with singularities $(\Fs^s, S, \nu^s)$ and $(\Fs^u, S, \nu^u)$ (with the same finite set of singularities $S = \{x_1, \ldots, x_m\}$) and an atlas of $C^\infty$ charts $\oldphi_k : U_k \to \C$ for $k = 1, \ldots, \ell$, $\ell > m$, satisfying the following properties: 
	\begin{enumerate}
		\item $f$ is differentiable, except on $S$.
		\item For each $x_k \in S$, $\Fs^s$ and $\Fs^u$ have the same number $p(k)$ of prongs at $x_k$. 
		\item The leaves of $\Fs^s$ and $\Fs^u$ intersect transversally at nonsingular points.
		\item Both measured foliations $\Fs^s$ and $\Fs^u$ are $f$-invariant.
		\item There is a constant $\lambda > 1$ such that 
		\[
		f(\Fs^s, \nu^s) = (\Fs^s, \nu^s/\lambda) \quad \textrm{and} \quad f(\Fs^u, \nu^u) = (\Fs^u, \lambda \nu^u).
		\]
		\item For each $k=1, \ldots, m$, we have $x_k \in U_k$, and $\oldphi_k : U_k \to \C$ satisfies: 
		\begin{enumerate}[label=(\alph*)]
			\item $\oldphi_k(x_k) = 0$ and $\oldphi_k(U_k) = D_{a_k}$ for some $a_k > 0$; 
			\item if $C$ is a curve leaf in $\Fs^s$, then the components of $C \cap U_k$ are mapped by $\oldphi_k$ to sets of the form 
			\[
			\left\{z \in \C: \re\left(z^{p/2}\right) = \mathrm{constant}\right\}\cap D_{a_k};
			\]
			\item if $C$ is a curve leaf in $\Fs^u$, then the components of $C \cap U_k$ are mapped by $\oldphi_k$ to sets of the form 
			\[
			\left\{z \in \C : \im\left(z^{p/2}\right) = \mathrm{constant} \right\} \cap D_{a_k};  
			\]
			\item the measures $\nu^s|U_k$ and $\nu^u|U_k$ are given by the pullbacks of $$\left|\re\left(dz^{p/2}\right)\right| = \left|\re\left(z^{(p-2)/2} dx \right)\right|$$
			and $$ \left|\im\left(dz^{p/2}\right)\right| = \left|\im\left(z^{(p-2)/2} dx \right)\right|$$
			under $\oldphi_k$, respectively. 
		\end{enumerate}
		\item For each $k > m$, we have: 
		\begin{enumerate}[label=(\alph*)]
			\item $\oldphi_k(U_k) = (0, b_k) \times (0,c_k) \subset \R^2 \approx \C$ for some $b_k, c_k > 0$; 
			\item If $C$ is a curve leaf in $\Fs^s$, then components of $C \cap U_k$ are mapped by $\oldphi_k$ to lines of the form
			\[
			\{z \in \C : \re \,z = \mathrm{constant}\} \cap \oldphi_k(U_k);
			\] 
			\item If $C$ is a curve leaf in $\Fs^u$, then components of $C \cap U_k$ are mapped by $\oldphi_k$ to lines of the form
			\[
			\{z \in \C : \im \,z = \mathrm{constant}\} \cap \oldphi_k(U_k);
			\] 
			\item the measures $\nu^s|U_k$ and $\nu^u|U_k$ are given by the pullbacks of $|\re\,dz|$ and $|\im \,dz|$ under $\oldphi_k$, respectively. 
		\end{enumerate}
	\end{enumerate}
	For $k = 1, \ldots, m$, we call the neighborhood $U_k \subset M$ described in part (6) of this definition a \emph{singular neighborhood}, and for $k > m$, we call $U_k$ a \emph{regular neighborhood}. (See Figure 2.)
\end{definition} 

\begin{figure}
	\centering
	\includegraphics[width=0.6\textwidth]{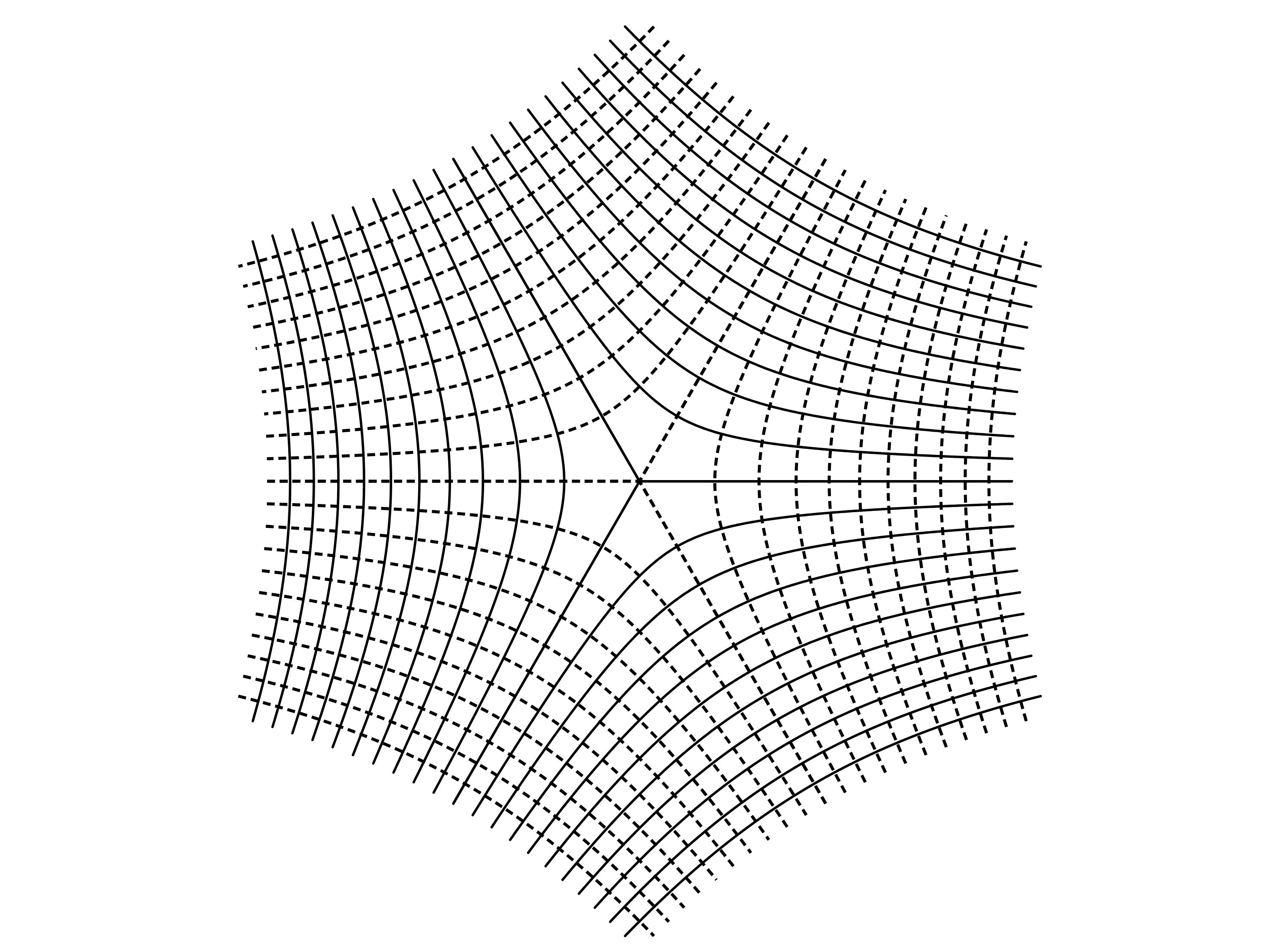}
	\caption{A singular neighborhood with a 3-pronged singularity. The solid lines and broken lines respectively represent the stable and unstable foliations $\Fs^s$ and $\Fs^u$, for example.}
\end{figure}

See Remarks 2.3 and 2.5 of \cite{VecPAD1} for a discussion on some of the technical intuition behind measured foliations and pseudo-Anosov homeomorphisms. 

\begin{proposition}\label{PAH-differential}
	Let $f : M \to M$ be a pseudo-Anosov homeomorphism. For $x \in M \setminus S$, the tangent space decomposes as a direct sum $T_x M = T_x \Fs^s(x) \oplus T_x \Fs^u(x)$, where $\Fs^s(x)$ and $\Fs^u(x)$ represent the curve containing $x$ in the respective foliation. In these coordinates, the differential of $f$ has the diagonal form $$Df_x(\xi^s, \xi^u) = \left( \xi^s/\lambda, \lambda \xi^u\right),$$ where $\xi^s$ and $\xi^u$ are nonzero vectors in $T_x \Fs^s(x)$ and $T_x \Fs^u(x)$, and $\lambda$ is the dilation factor for $f$. 
\end{proposition}

\begin{proof}
	This follows immediately from the definition of pseudo-Anosov diffeomorphisms after a calculation in coordinates (see Remark 2.5 of \cite{VecPAD1}). 
\end{proof}

\begin{proposition}[\cite{FS79}]\label{PAH measure}
	A pseudo-Anosov surface homeomorphism $f : M \to M$ preserves a smooth invariant probability measure $\nu$ defined locally as the product of $\nu^s$ on $\Fs^u$-leaves with $\nu^u$ on $\Fs^s$-leaves. In any coordinate chart of $M$, this probability measure $\nu$ has a density with respect to the measure induced by the Lebesgue measure on $\R^2$, and this density vanishes at singularities. 
\end{proposition}

\begin{proposition}[\cite{FS79}]\label{pseudo-Anosov Markov}
	Every pseudo-Anosov homeomorphism of a surface $M$ admits a finite Markov partition of arbitrarily small diameter. The system $(M, f, \nu)$ has the Bernoulli property via the symbolic representation for this Markov partition (see Definition \ref{def:bernoulli} below), where $\nu$ is the measure in the preceding proposition.
\end{proposition}

\subsection{Ergodic properties}

For the reader's convenience, we describe here the thermodynamic and ergodic properties that we will refer to throughout the paper. Throughout the following, $T : X \to X$ will be a measurable and invertible transformation preserving a measure $\mu$ on $X$. 

\begin{definition}\label{def:bernoulli}
	The transformation $(T,\mu)$ has the \emph{Bernoulli property} if there is a Lebesgue space $(Y,\nu)$ for which $(T,\mu)$ is metrically isomorphic to the corresponding Bernoulli shift $\sigma : Y^\Z \to Y^\Z$, where $Y^\Z$ is endowed with the measure $\nu^{\otimes \Z}$. 
\end{definition}

\begin{definition}
	Suppose $\Hs_1$ and $\Hs_2$ are two classes of real-valued functions on $X$ (also called \emph{observables} on $(T,\mu)$). For $h_1 \in \Hs_1$ and $h_2 \in \Hs_2$, the $n^{\textrm{th}}$ \emph{correlation} between the two observables is 
	\[
	\mathrm{Cor}_n(h_1, h_2) := \int h_1(T^n(x)) h_2(x) d\mu(x) - \int h_1 \,d\mu \int h_2\,d\mu. 
	\]
	The transformation $(T,\mu)$ has \emph{exponential decay of correlations} if there is a constant $\gamma_0 > 0$ for which for any $h_1 \in \Hs_1, h_2 \in \Hs_2$, 
	\[
	|\mathrm{Cor}_n(h_1, h_2)| \leq Ce^{-\gamma_0 n},
	\] 
	where $C_0 = C_0(h_1, h_2)$ is independent of $n$.
	
	The transformation $T$ has \emph{polynomial upper} or \emph{lower bound on correlations} with respect to $\Hs_1$ and $\Hs_2$ if, respectively, there is a number $\gamma_1 > 0$ for which for any $h_1 \in \Hs_1, h_2 \in \Hs_2$, 
	\[
	|\mathrm{Cor}_n(h_1, h_2)| \leq C_1 n^{-\gamma_1};
	\]
	or, if there is a number $\gamma_2 > 0$ for which for any $h_1 \in \Hs_1, h_2 \in \Hs_2$, 
	\[
	|\mathrm{Cor}_n(h_1, h_2)| \geq C_2 n^{-\gamma_2},
	\]
	where in each case $C_1$ and $C_2$ are constants independent of $n$ (but may depend on $h_1, h_2$). 
\end{definition}

\begin{definition}
	The system $(T,\mu)$ satisfies the \emph{Central Limit Theorem (CLT)} with respect to a class $\Hs$ of observables if there is a $\sigma > 0$ such that for any $h \in \Hs$ with $\int h \,d\mu = 0$, we have
	\[
	\mu \left\{ x \in X : \frac{1}{\sqrt n} \sum_{i=0}^{n-1} \left(h(T^i(x)) - \int h \,d\mu\right)< t\right\} \xrightarrow{n \to \infty} \frac{1}{\sigma\sqrt{2\pi}} \int_{-\infty}^{t} e^{-\tau^2/2\sigma^2}\,d\tau.
	\]
\end{definition}

\begin{definition}
	The transformation $(T,\mu)$ has \emph{polynomial large deviation} with respect to a class $\Hs$ of observables on $X$ if there is a $\beta > 0$ such that for any $h \in \Hs$, any $\epsilon > 0$, and any sufficiently large $n > 0$, we have \[
	\mu\left\{x \in X : \left| \frac 1 n \sum_{i=0}^{n-1} h \left(T^i(x)\right) - \int h \, d\mu\right| > \epsilon  \right\} < Kn^{-\beta},
	\]
	where $K = K(h,\epsilon) > 0$ is a constant independent of $n$.
\end{definition}

\section{Main results}

We now state our main result. The exponents $\gamma_1, \gamma_2$ for the polynomial decay depend on parameters related to the slowdown procedure for the pseudo-Anosov homeomorphism $f$. The specific values are $\gamma_1 = \gamma' - 2$ and $\gamma_2 = \gamma - 2$, where $\gamma$ and $\gamma'$ are given in \eqref{gamma-defs}. 

\begin{theorem}\label{Main 1}
	Let $f : M \to M$ be a pseudo-Anosov homeomorphism of a compact orientable Riemannian surface $M$. There is a $C^{2+\epsilon}$ diffeomorphism $g : M \to M$, $\epsilon > 0$ depending on $f$, that is topologically conjugate and $C^0$-close to $f$. Furthermore, there are numbers $\beta > 0$, $\eta > 0$, and $\gamma_2 > \gamma_1 > 0$ for which the map $g$ also satisfies the following properties: 
	\begin{enumerate}[label=\emph{(\arabic*)}]
		\item $g$ preserves a probability measure $\mu_1$ that is equivalent to the Riemannian area of $M$. 
		\item $g$ has nonzero Lyapunov exponents at $m$-a.e. $x$. 
		\item $g$ has the Bernoulli property with respect to $m$. 
		\item $g$ has polynomial upper and lower bounds on the correlations with respect to $m$ and the set of $\eta$-H\"older continuous functions for some $\eta > 0$. More precisely: 
		\begin{enumerate}[label=\emph{(\alph*)}]
			\item for any $h_i \in C^\eta$, $i = 1,2$, $$\mathrm{Cor}_n(h_1, h_2)| \leq C_1 n^{-\gamma_1},$$ where $C_1 = C_1(\|h_1\|_{C^\eta}, \|h_2\|_{C^{\eta}})$; 
			\item there is a nested sequence of subsets $\{M_j\}_{j \geq 1}$ that exhausts $M$ for which if $h_1, h_2 \in C^\eta$ are such that $\int h_1 \,dm \int h_2\,dm > 0$ and $\mathrm{supp}(h_i)\subset M_j$ for some $j$, for $i = 1,2$,
			\[
			|\mathrm{Cor}_n(h_1, h_2)| \geq C_2 n^{-\gamma_2},
			\]
			where $C_2 = C_2(\|h_1\|_{C^\eta}, \|h_2\|_{C^\eta})$. 
		\end{enumerate}
		\item $g$ satisfies the CLT with respect to the class of observables $C^\eta_0 := \{h \in C^\eta : \int h \,dm = 0\}$, with $\sigma = \sigma(h)$ given by 
		\[
		\sigma^2 = -\int h^2\,dm + 2\sum_{n=0}^\infty \int h \cdot(h \circ f^n)\,dm
		\]
		where $\sigma > 0$ iff $h$ is not cohomologous to zero (i.e., $h \circ f \neq h' \circ f - h'$ for any measurable function $h'$).
		\item $g$ has polynomial large deviations with respect to the class $C^\eta$ of observables with the constant $K = K(\|h\|_{C^\eta})\epsilon^{-2\beta}$. Furthermore, for an open and dense subset of observables in $C^\eta$ and sufficiently small $\epsilon > 0$, 
		\[
		n^{-\beta} < m \left( \left| \frac 1 n \sum_{j=0}^{n-1} h(f^j(x)) - \int h \, dm\right| > \epsilon\right)
		\]
		for infinitely many $n$; 
		\item $g$ has a unique measure of maximal entropy, with respect to which $g$ has the Bernoulli property, nonzero Lyapunov exponents almost everywhere, exponential decay of correlations, and the Central Limit Theorem with respect to H\"older observables. 
	\end{enumerate}
\end{theorem}

\section{Construction of the smooth pseudo-Anosov model}

As shown in Section 2.4 of \cite{GerbKatPA}, pseudo-Anosov homeomorphisms as we've defined them are not smooth at the singularities. We construct a smooth realization of the pseudo-Anosov map, adapted from the procedures in both \cite{GerbKatPA} and \cite{Kat79}. The resulting map $g : M \to M$ will be a $C^{2+\epsilon}$ diffeomorphism whose differential at the singularities is the identity. 

Before proceeding with the construction, we point out that some literature refers to the maps defined in Definition \ref{PAH-def} as ``pseudo-Anosov diffeomorphisms'', despite the fact that these maps are not differentiable at the singularities. To avoid any confusion, we reserve the word ``diffeomorphism'' only for those maps that are differentiable on all of $M$, and use the phrase ``pseudo-Anosov homeomorphism'' for the maps described in Definition \ref{PAH-def}. 

\subsection{Construction of $g$}\label{sec:construction}

Let $x_k$ be a singularity of $f$, let $p = p(x_k)$ be the number of prongs at this singularity, and let $\oldphi_k : U_k \to \C$ be the chart described in part (6) of Definition (\ref{PAH-def}). The \emph{stable} and \emph{unstable prongs} at $x_k$ are the leaves $P^s_{kj}$ and $P^u_{kj}$, $j = 0, \ldots, p-1$ of $\Fs^s$ and $\Fs^u$, respectively, whose endpoints meet at $x_k$. Locally, they are given by:
\begin{align*}
P^s_{kj} &= \oldphi_k^{-1} \left\{\rho e^{i\tau} : 0 \leq \rho < a_k, \: \tau = \frac{2j+1}p \pi \right\}, \\
\textrm{and}\quad P^u_{kj} &= \oldphi_k^{-1} \left\{\rho e^{i\tau} : 0 \leq \rho < a_k, \: \tau = \frac{2j}p \pi \right\}.
\end{align*}
Since $f : M \to M$ is a homeomorphism, $f$ permutes the singularities. Therefore, after taking a suitable iterate, assume the singularities are fixed points, and moreover, assume $f(P^s_{kj}) \subseteq P^s_{kj}$ for all $j = 0, \ldots, p-1$. Furthermore, we define the \emph{stable} and \emph{unstable sectors} at $x_k$ to be the regions in $U_k$ bounded by the stable (resp. unstable) prongs: 
\begin{align*}
S^s_{kj} &= \oldphi_k^{-1} \left\{\rho e^{i\tau} : 0 \leq \rho < a_k, \: \frac{2j-1}p \pi \leq \tau \leq \frac{2j+1}p \pi \right\}, \\
\textrm{and}\quad S^u_{kj} &= \oldphi_k^{-1} \left\{\rho e^{i\tau} : 0 \leq \rho < a_k, \: \frac{2j}p \pi \leq \tau \leq \frac{2j+2}p \pi \right\}.
\end{align*}
Assume, after taking a suitable iterate, that $f(S_{kj}^u) \subset S_{kj}^u$ and $f(S_{kj}^s) \subset S_{kj}^s$. 

Our strategy will be to apply a ``slow-down'' of the trajectories in each stable sector $S^s_{kj}$, followed by a change of coordinates ensuring the resulting diffeomorphism $g$ preserves the measure induced by a convenient Riemannian metric. 

Let $F : \C \to \C$ be the map $s_1 + is_2 \mapsto \lambda s_1 + is_2/\lambda$. Note $F$ is the time-1 map of the vector field $V$ given by 
\begin{equation}\label{eq:linear-vector-field}
\dot{s}_1 = (\log\lambda)s_1, \quad \dot{s}_2 = -(\log\lambda) s_2.
\end{equation}
Let $0 < \rho_1 < \rho_0 < \min\{a_1, \ldots, a_\ell\}=:a^*$, and define $ r_0$ and $ r_1$ by $r_j = (2/p)\rho^{p/2}_j$ for $j=0,1$ and for each $p = p(k)$. Also let $\tilde a = (2/p)(a^*)^{p/2}$. Assume $\rho_0, \rho_1$ are chosen so that
\begin{equation}\label{eq:F-neighborhood-containments}
	D_{r_1} \subset F(D_{r_0}), \quad F(D_{r_1}) \cup F^{-1}(D_{r_0}) \subset D_{\tilde a}.
\end{equation}
We also assume $\rho_0$ is chosen to be small enough so that the open neighborhood $\mathcal U_0 := \union_{k=1}^m \oldphi_k^{-1}\left(D_{\rho_0}\right)$ of the set $S$ of singularities is disjoint from the open set $\union_{k={m+1}}^\ell \oldphi_k^{-1}\left(D_{a_k}\right) = \union_{k=m+1}^\ell U_k$. 

Let $\alpha \in (0,1)$ be a uniform constant. For each $p$-pronged singularity, define a ``slow-down'' function $\Psi_p = \Psi_{p,\alpha}$ on the interval $[0, \infty)$ so that: 
\begin{enumerate}[label=(\arabic*)]
	\item $\Psi_{p}(u) = \left( \frac p 2 \right)^{2\alpha}u^{\alpha}$	for $u \leq r_1^2$; 
	\item $\Psi_{p}$ is $C^\infty$ except at $0$; 
	\item $\dot\Psi_p(u) \geq 0$ for $u > 0$; 
	\item $\Psi_p(u) = 1$ for $u \geq r_0^2$.
\end{enumerate}
Consider the vector field $\hat V_{p}$ on $D_{ r_0} \subset \C$ defined by 
\begin{equation}\label{eq:slowdown vector field}
\dot s_1 = (\log \lambda) s_1 \Psi_{p} \left(s_1^2 + s_2^2\right) \quad \textrm{and} \quad \dot s_2 = -(\log \lambda) s_2 \Psi_{p}\left(s_1^2 + s_2^2\right).
\end{equation}
Let $G_{p}$ be the time-1 map of the vector field $\hat V_{p}$. Assume $\rho_1$ is chosen to be small enough so that $G_p = F$ on a neighborhood of the boundary of $D_{r_0}$.
%
%

Now let $\phi : \C\to \C$ be given by
\begin{equation}\label{eq:mass-push-coord-change}
\phi(z) = A\left(\int_0^{|z|^2} \frac{du}{\Psi_p(u)}\right)^{p/4}\frac{z}{|z|},
\end{equation} 
where $A > 0$ is defined by 
\begin{equation}\label{eq:A-def}
A = \left( (1-\alpha)\left( \frac p 2 \right)^{2\alpha}\right)^{p/4}.
\end{equation}
In particular, observe that for $0 < |z| < r_1$, we have $\Psi_p(u) = \left( \frac p 2 \right)^{2\alpha} u^\alpha$, and so for $|z|=r<r_1$: 
\begin{equation}\label{eq:mass-push-near-0}
\phi(re^{i\theta}) = A\left( \int_0^{r^2} \left( \frac p 2 \right)^{-2\alpha} u^{-\alpha} \,du\right)^{p/4} e^{i\theta} = r^{p(1-\alpha)/2}e^{i\theta}.
\end{equation}
Therefore near 0, denoting $\tilde r e^{i\tilde{\theta}} = \phi(re^{i\theta})$, the coordinates $(r,\theta)$ and $(\tilde r, \tilde\theta)$ are related by 
\[
\tilde r = r^{p(1-\alpha)/2}, \quad \textrm{and} \quad \tilde\theta = \theta.
\]

For each singularity $x_k$, let $\tilde a_k = (2/p) a_k^{p/2}$, and define the coordinate change $\Phi_{kj} : \oldphi_k S^s_{kj} \to \left\{z : \re z \geq 0 \right\} \cap D_{\tilde a_k}$ by 
\begin{equation}\label{eq:sector-coord-map}
\Phi_{kj}(z) = \Phi_{kj}(\rho e^{i\tau}) = (-1)^j\frac 2 p z^{p/2} = \frac 2 p \rho^{p/2} e^{i\tau \frac p 2 + ij\pi} = \tilde r e^{i\tilde\theta}.
\end{equation}
Observe, therefore, that the coordinates $(\rho, \tau)$ and $(\tilde r, \tilde \theta)$ are related by 
\[
\rho = \left( \frac p 2 \tilde r\right)^{2/p} \quad \textrm{and} \quad \tau = \frac 2 p \tilde{\theta} - \frac{2j\pi}p 
\]

Define $g : M \to M$ by $g(x) = f(x)$ for $x \not\in \Us_0$ and meanwhile for $1 \leq k \leq m$, $1 \leq j \leq p(k)$, define $g$ on each sector $S^s_{kj}\cap\oldphi_k^{-1}\left(D_{\rho_0}\right)$ by 
\begin{equation}\label{eq:g-def}
g = \left(\phi^{-1} \circ \Phi_{kj} \circ \oldphi_k\right)^{-1} \circ G_p\circ \left( \phi^{-1}\circ \Phi_{kj} \circ\oldphi_k\right).
\end{equation}
Note that $g = f$ in $\oldphi_k^{-1} \left( D_{a_k} \setminus D_{\rho_0}\right)$, and therefore it follows from \eqref{eq:F-neighborhood-containments} that
\begin{equation}\label{eq:g-neighborhood-containments}
	\oldphi_k^{-1}\left(D_{\rho_1}\right) \subset g\left( \oldphi_k^{-1}(D_{\rho_0})\right), \quad g\left(\oldphi_k^{-1}(D_{r_1})\right) \cup g^{-1}\left(\oldphi_k^{-1}(D_{r_0})\right) \subset \oldphi_k^{-1}\left(D_{\tilde a}\right).
\end{equation}
\begin{remark}
	In the original smooth pseudo-Anosov realization constructed in \cite{GerbKatPA}, the exponent they chose is $\alpha = (p-2)/p$, in which case one can compute $\phi = \mathrm{id}$. 
\end{remark}

\subsection{Smoothness and area invariance of $g$}

We now show that $g$ is a $C^{2+\epsilon}$ diffeomorphism on $M$ and preserves a smooth invariant measure. Let $x_k \in M$ be a singularity of $g$. Consider the vector field $V$ given by \eqref{eq:linear-vector-field} defined on $D_{r_1} = (\phi^{-1} \circ \Phi_{kj})(D_{\rho_1})$, and let $\Omega = ds_1 \wedge ds_2 = rdr\wedge d\theta$ be the Lebesgue area form. Observe that $V$ is Hamiltonian with respect to $\Omega$, with Hamiltonian function $H(s_1, s_2) = s_1 s_2 \log\lambda$. Define the area form $\hat\Omega_p$ by
\[
(\hat\Omega_p)_{(s_1, s_2)} = \frac{ds_1 \wedge ds_2}{\Psi_p(s_1^2 + s_2^2)} = \frac{rdr\wedge d\theta}{\Psi_p(r^2)}.
\]
Note the vector field $\hat V_p$ defined by \eqref{eq:slowdown vector field} is Hamiltonian with respect to $\hat\Omega_p$, with Hamiltonian function $H$. Finally let $V_p$ be the (continuous) vector field on $D_{a_k} \subset \C$ given by $(\Phi_{kj}^{-1} \circ \phi)_* \hat V_p$, and let $\Omega_p = (\phi^{-1} \circ \Phi_{kj})^* \hat\Omega_p$. Note $V_p$ is Hamiltonian with respect to $\Omega_p$, with Hamiltonian function $H_p := H \circ \phi^{-1} \circ \Phi_{kj}$. 

\begin{lemma}\label{lem:lebesgue-near-sings}
	Near the origin, $\Omega_p$ is a constant times Lebesgue area in $D_{a_k}$.
\end{lemma}

\begin{proof}
	Note that for $\rho > 0$ sufficiently small, the function $\left(\phi^{-1} \circ \Phi_{kj}\right)(\rho e^{i\tau}) = r e^{i\theta}$ satisfies 
	\begin{equation}\label{eq:rho-r-coord-change}
		\begin{aligned}
	re^{i\theta} &= \left(\phi^{-1} \circ \Phi_{kj}\right)(\rho e^{i\tau}) \\
	&= \phi^{-1} \left( \frac 2 p \rho^{p/2} e^{i\tau\frac p 2 + i j\pi}\right) \\
	&=\left( \frac 2 p \right)^{2/p(1-\alpha)}\rho^{1/(1-\alpha)} e^{i\tau \frac p 2 +ij\pi},
	\end{aligned}
	\end{equation}
	and so the coordinates $(r,\theta)$ and $(\rho, \tau)$ are related by 
	\[
	r = \left( \frac 2 p \right)^{2/p(1-\alpha)}\rho^{1/(1-\alpha)}, \quad \theta = \frac p 2 \tau + j\pi.
	\]
	It follows that: 
	\[
	dr = \frac{1}{1-\alpha} \left( \frac 2 p \right)^{2/p(1-\alpha)} \rho^{\alpha/(1-\alpha)} d\rho \quad \textrm{and} \quad d\theta = \frac p 2 d\tau.
	\]
	So, since in polar coordinates we can write $\hat\Omega_p = \frac{1}{\Psi_p(r^2)} rdr\wedge d\theta$, for $\rho e^{i\tau}$ sufficiently near 0, we have:
	\begin{equation}\label{eq:lebesgue-preserving-last-step}
	\begin{aligned}
	\Omega_p &= \left( \phi^{-1} \circ \Phi_{kj}\right)^* \hat\Omega_p \\
	&= \left( \phi^{-1} \circ \Phi_{kj}\right)^* \left( \frac{rdr\wedge d\theta}{\Psi_p(r^2)}\right) \\
	&= \frac{1}{\Psi_p\left( \left( \frac 2 p \right)^{4/p(1-\alpha)} \rho^{2/(1-\alpha)}\right)} \times \left( \left( \frac 2 p \right)^{2/p(1-\alpha)}\rho^{1/(1-\alpha)}\right)\\
	&\quad \times \left(\frac{1}{1-\alpha} \left( \frac 2 p \right)^{2/p(1-\alpha)} \rho^{\alpha/(1-\alpha)} d\rho\right) \wedge \left( \frac p 2 d\tau\right)\\
	&= \left(\left( \frac p 2 \right)^{-2\alpha}  \left( \frac p 2 \right)^{\frac{4\alpha}{p(1-\alpha)}} \rho^{-\frac{2\alpha}{1-\alpha}}\right) \times \left(  \left( \frac p 2 \right)^{-\frac 2{p(1-\alpha)}} \rho^{\frac 1{1-\alpha}}\right)\\
	&\quad \times \left( \frac 1{1-\alpha} \left( \frac p 2 \right)^{1-\frac{2}{p(1-\alpha)}} \rho^{\frac{\alpha}{1-\alpha}}\right) d\rho \wedge d\tau \\
	&= \frac 1{1-\alpha} \left( \frac p 2 \right)^{1-\frac 4 p - 2\alpha} \rho d\rho \wedge d\tau.
	\end{aligned}
	\end{equation}
	Since $\rho d\rho \wedge d\tau$ is the Lebesgue area in $D_{a_k}$, we've proven the lemma. 
\end{proof}

\begin{remark}
	In the original smooth pseudo-Anosov realization constructed in \cite{GerbKatPA}, the exponent they chose is $\alpha = (p-2)/p$, in which case one can compute that the constant in front of $\rho d\rho \wedge d\tau$ in the final equality of \eqref{eq:lebesgue-preserving-last-step} is 1, and the area is precisely Lebesgue area. 
\end{remark}

Recall $V_p = (\Phi_{kj}^{-1} \circ \phi)_* \hat V_p$, where $\hat V_p$ is given by \eqref{eq:slowdown vector field}. Note $\diver_{\Omega} V = 0$, and it follows that $\diver_{\hat{\Omega}_p} \hat V_p = 0$, and so $\diver_{\Omega_p} V_p = 0$ in a neighborhood of each singularity. Since $g$ is the time-1 map of $V_p$ on $M$, one can use a partition of unity on $(U_k, \oldphi_k)_{1 \leq k \leq \ell}$ and the coordinate representation of $g$ in each chart to prove:

\begin{proposition}\label{prop:SRB}
	The map $g : M \to M$ preserves a smooth invariant measure $\mu_1$ that is equivalent to the Riemannian area on $M$.
\end{proposition}

We next show $g$ is $C^{2+\epsilon}$. To do this, we need the following technical result: 

\begin{lemma}\label{lem:polynomial-Holder}
	Suppose $f(t_1, t_2) = C|(t_1, t_2)|^{\beta} Q(t_1, t_2)$, where $|(t_1, t_2)| = \sqrt{t_1^2 + t_2^2}$ and $Q : \R^2 \to \R$ is a polynomial whose terms are all of order $p$. That is, 
	\[
	Q(t_1, t_2) = \sum_{j=0}^p A_j t_1^j t_2^{p-j}, \quad A_j \in \R.
	\]
	Then, for $j = 1,2$,
	\begin{equation}\label{eq:polynomial-lemma-1}
		\frac{\del f}{\del t_j} = C|(t_1, t_2)|^{\beta-2} Q_1(t_1, t_2),
	\end{equation}
	where $Q_1$ is a polynomial whose terms are all of order $p+1$. In particular, inductively, it follows that for every $k \geq 1$ and every $0 \leq \ell \leq k$,
	\begin{equation}\label{eq:polynomial-lemma-k}
		\frac{\partial^k f}{\partial t_1^\ell \partial t_2^{k-\ell}} = C|(t_1, t_2)|^{\beta - 2k}Q_k(t_1, t_2)
	\end{equation}
	where $Q_k$ is a poynomial whose terms are all of degree $p+k$. 
\end{lemma}

\begin{proof}
	If $Q(t_1, t_2)$ has monomial terms all of degree $p$, then $Q_{t_j}(t_1, t_2)$ has terms all of degree $p-1$. Meanwhile, 
	\[
	\frac{\del}{\del t_j} |(t_1, t_2)|^\beta = \frac{\del}{\del t_j} \left( t_1^2 + t_2^2\right)^{\beta/2} = \beta t_j \left( t_1^2 + t_2^2\right)^{(\beta-2)/2} = \beta t_j |(t_1, t_2)|^{\beta-2}
	\]
	If $f(t_1, t_2) = C|(t_1, t_2)|^\beta Q(t_1, t_2)$, it follows that
	\begin{align*}
	\frac{\del f}{\del t_j} &= C\left( t_j |(t_1, t_2)|^{\beta-2} Q(t_1, t_2) + |(t_1, t_2)|^\beta Q_{t_j}(t_1, t_2)\right)\\
	&= C|(t_1, t_2)|^{\beta-2} \left( t_j Q(t_1, t_2) + (t_1^2 + t_2^2) Q_{t_j}(t_1, t_2)\right).
	\end{align*}
	\eqref{eq:polynomial-lemma-1} now follows with $Q_1(t_1, t_2) = t_j Q(t_1, t_2) + (t_1^2 + t_2^2) Q_{t_j}(t_1, t_2)$, and \eqref{eq:polynomial-lemma-k} follows by induction. 
\end{proof}

\begin{proposition}
	$H_p = H \circ \phi^{-1} \circ \Phi_{kj}$ is at least $C^{2+\epsilon}$, where $\epsilon = \frac{2}{1-\alpha} - \left\lfloor \frac{2}{1-\alpha}\right\rfloor$. 
\end{proposition}

\begin{proof}
	Note $H(s_1, s_2) = s_1 s_2\log\lambda $ in polar coordinates is
	\[
	H(re^{i\theta}) = (\log\lambda )r^2 \cos\theta\sin\theta = \frac 1 2 (\log\lambda )r^2 \sin(2\theta).
	\]
	It follows from \eqref{eq:rho-r-coord-change} that for $\rho e^{i\tau} = t_1 + it_2$, we have:
	\begin{equation}\label{eq:Hp}
		\begin{aligned} 
		H_p(\rho e^{i\tau}) &= \frac 1 2 (\log\lambda)\left( \frac 2 p \right)^{4/p(1-\alpha)} \rho^{2/(1-\alpha)} \sin(\tau p)\\
		&= \frac 1 2 (\log\lambda)  \left( \frac 2 p \right)^{4/p(1-\alpha)} \rho^{\frac{2-p(1-\alpha)}{1-\alpha}} \rho^p \sin(\tau p) \\
		&= \frac 1 2 (\log\lambda) \left( \frac 2 p \right)^{4/p(1-\alpha)} |t_1 + it_2|^{\frac{2}{1-\alpha} - p} \mathrm{Im}(z^p).
		\end{aligned} 
	\end{equation}
	Since $\mathrm{Im}(z^p)$ is a polynomial in $t_1$ and $t_2$ whose monomial terms are all of order $p$, Lemma \ref{lem:polynomial-Holder} gives us that for $k \geq 1$, $0 \leq \ell \leq k$, 
	\begin{equation}\label{eq:partial-Hp}
	\frac{\del^k H_p}{\del t_1^\ell \del t_2^{k-\ell}} = \frac 1 2 (\log\lambda)  \left( \frac 2 p \right)^{4/p(1-\alpha)} |t_1 + it_2|^{\frac{2}{1-\alpha} - p - 2k} Q_k(t_1, t_2),
	\end{equation}
	where $Q_k$ is a polynomial whose monomial terms are of degree $p+k$. In other words, 
	\[
	Q_k(t_1, t_2) = Q_k(\rho e^{i\tau}) = \rho^{p+k} h(\tau),
	\]
	where $h : [0,2\pi] \to \R$ is a continuous and bounded function. It follows from \eqref{eq:partial-Hp} that
	\begin{equation}\label{eq:partial-Hp-condensed}
		\partial^\ell_{t_1} \partial^{k-\ell}_{t_2} H_p := \frac{\partial^k H_p}{\partial t_1^\ell \partial t_2^{k-\ell}}(\rho e^{i\tau}) = B\rho^{\frac{2}{1-\alpha} - p - 2k + (p+k)} h(\tau) = B\rho^{\frac{2}{1-\alpha} - k} h(\tau),
	\end{equation}
	where $B > 0$ is a constant. This function is continuous on $\C$ as long as $k < \frac{2}{1-\alpha}$. Note $\frac{2}{1-\alpha} > 2$ since $0 < \alpha < 1$. For $k = \left\lfloor \frac{2}{1-\alpha}\right\rfloor$, it follows that $H_p$ is $C^{k+\epsilon}$, $\epsilon = \frac{2}{1-\alpha} - \left\lfloor \frac{2}{1-\alpha}\right\rfloor$. 
\end{proof}

Since the vector field $V_p$ is Hamiltonian with respect to Lebesgue area with Hamiltonian function $H_p$, it follows that $V_p$ is $C^{2+\epsilon}$, and thus the map $g : M \to M$ is $C^{2+\epsilon}$ (note $g$ is $C^\infty$ away from the singularities). 

\subsection{Other topological properties} 

The smooth realization $g$ of a pseudo-Anosov homeomorphism $f$ is adapted from a smooth realization of pseudo-Anosov homeomorphisms first described in \cite{GerbKatPA}. In this construction, the slow-down exponent $\alpha$ in the definition of $\Psi_p$ is taken to be $\alpha = (p-2)/p$. It follows that the homeomorphism $\phi : \C \to \C$ is the identity and the Hamiltonian function $H_p$ of $V_p$ is a constant times $\mathrm{Im}(z^p)$, i.e., a polynomial (see \eqref{eq:Hp}), and hence $V_p$ is analytic. Therefore the smooth pseudo-Anosov model in \cite{GerbKatPA} is analytic, not just $C^{2+\epsilon}$. However, using similar arguments to Section 6.4 of \cite{BaPeNUH}, $C^{2+\epsilon}$ is sufficient regularity to prove the following:

\begin{proposition}\label{prop:topological-properties}
	The smooth pseudo-Anosov realization $g : M \to M$ defined by \eqref{eq:g-def} has the following properties: 
	\begin{enumerate}[label=(\alph*)]
		\item $g$ is topologically conjugate to the linear pseudo-Anosov homeomorphism $f$, via a continuous (but not $C^1$) conjugacy $h$ isotopic to the identity. 
		\item For any $\epsilon > 0$, one can choose $\alpha, \rho_0,$ and $\rho_1$ in the construction of $g$ so that $\|f - g\|_{C^0} < \epsilon$. 
		\item The map $g$ admits two invariant distributions $x \mapsto E^u(x), E^s(x)$, which are continuous on $M$ except at the singularities. At $\mu_1$-a.e. $x \in M$ (where $\mu_1$ is the measure in Proposition \ref{prop:SRB}), $g$ admits two nonzero Lyapunov exponents: one negative exponent in the direction of $E^s(x)$, and one positive exponent in the direction of $E^u(x)$. 
		\item The map $g$ admits two invariant foliations with singularities of $M$, which are the images under the conjugating homeomorphism $h$ of the foliations with singularities $\Fs^s$ and $\Fs^u$ associated to the pseudo-Ansoov homeomorphism $f$. 
		\item The map $g$ admits a finite Markov partition, given by the image of the Markov partition of $f$ under the conjugating homeomorphism $h$. 
	\end{enumerate}
\end{proposition}


Finally, in the case when the slowdown exponent is $\alpha = (p-2)/p$, it is shown in \cite{VecPAD1} that the geometric $t$-potentials $\phi_t(x) = -t\log\left| Dg|_{E^u(x)}\right|$ admit unique equilibrium states for $t_0 < t < 1$, $t_0 < 0$, which includes a unique measure of maximal entropy. Furthermore, these equilibrium states have exponential decay of correlations and the Central Limit Theorem with respec to H\"older-continuous potentials. Using identical techniques in \cite{PSZ17} and \cite{VecPAD1}, as well as results from \cite{PSZTowers} and \cite{SZ-inducing}, this result extends verbatim to pseudo-Anosov smooth realizations with arbitrary slowdown exponents $0 < \alpha < 1$:

\begin{proposition}\label{prop:PAD-thermodynamics}
	The following hold for the pseudo-Anosov smooth realization $g$:
	\begin{enumerate}
		\item Given any $t_0 < 0$, we may take $\rho_0 > 0$ in the construction of $g$ so that for any $t \in (t_0, 1)$, there is a unique equilibrium measure $\mu_t$ associated to $\phi_t$. This equilibrium measure has nonzero Lyapunov exponents almost everywhere, exponential decay of correlations and satisfies the Central Limit Theorem with respect to a class of functions containing all H\"older continuous functions on $M$, and is Bernoulli. Additionally, the pressure function $t \mapsto P_g(\phi_t)$ is real analytic in the open interval $(t_0, 1)$.
		\item For $t=1$, there are two classes of equilibrium measures associated to $\phi_1$: convex combinations of Dirac measures concentrated at the singularities, and a unique invariant SRB measure $\mu$.
		\item For $t > 1$, the equilibrium measures associated to $\phi_t$ are precisely the convex combinations of Dirac measures concentrated at the singularities.  
	\end{enumerate}
\end{proposition}

\section{Pseudo-Anosov Diffeomorphisms are Young Diffeomorphisms}

\subsection{Young diffeomorphisms} The proof of Theorem \ref{Main 1} relies on recent results on the thermodynamics of Young diffeomorphisms. In this section, we define Young diffeomorphisms and describe some of their thermodynamic properties. The following description of Young diffeomorphisms is discussed in Section 4 of \cite{PSZ17} and Section 6 of \cite{VecPAD1}, and is printed here for the reader's convenience. 

Given a $C^{1+\alpha}$ diffeomorphism $f$ on a compact Riemannian manifold $M$, we call an embedded $C^1$ disc $\gamma \subset M$ an \emph{unstable disc} (resp. \emph{stable disc}) if for all $x, y \in \gamma$, we have $d(f^{-n}(x), f^{-n}(y)) \to 0$  (resp. $d(f^n(x), f^n(y)) \to 0$) as $n \to +\infty$. A collection of embedded $C^1$ discs $\Gamma = \{\gamma_i\}_{i \in \mathcal I}$ is a \emph{continuous family of unstable discs} if there is a Borel subset $K^s \subset M$ and a homeomorphism $\Phi : K^s \times D^u \to \union_i \gamma_i$, where $D^u \subset \R^d$ is the closed unit disc for some $d < \dim M$, satisfying: 
\begin{itemize}
	\item The assignment $x \mapsto \Phi|_{\{x\} \times D^u}$ is a continuous map from $K^s$ to the space of $C^1$ embeddings $D^u \hookrightarrow M$, and this assignment can be extended to the closure $\overline{K^s}$; 
	\item For every $x \in K^s$, $\gamma = \Phi(\{x\} \times D^u)$ is an unstable disc in $\Gamma$.
\end{itemize}
Thus the index set $\mathcal I$ may be taken to be $K^s \times \{0\} \subset K^s \times D^u$. We define \emph{continuous families of stable discs} analogously. 

A subset $\Lambda \subset M$ has \emph{hyperbolic product structure} if there is a continuous family $\Gamma^u = \{\gamma^u_i\}_{i \in \mathcal I}$ of unstable discs and a continuous family $\Gamma^s = \{\gamma^s_j\}_{j \in \mathcal J}$ of stable discs such that
\begin{itemize}
	\item $\dim \gamma^u_i + \dim\gamma^s_j = \dim M$ for all $i,j$; 
	\item the unstable discs are transversal to the stable discs, with an angle uniformly bounded away from 0; 
	\item each unstable disc intersects each stable disc in exactly one point; 
	\item $\Lambda = \big( \union_i \gamma^u_i\big) \cap \big(\union_j \gamma^s_j \big)$. 
\end{itemize}

A subset $\Lambda_0 \subset \Lambda$ with hyperbolic product structure is an \emph{s-subset} if the continuous family of unstable discs defining $\Lambda_0$ is the same as the continuous family of unstable discs for $\Lambda$, and the continuous family of stable discs defining $\Lambda_0$ is a subfamily $\Gamma_0^s$ of the continuous family of stable discs defining $\Gamma_0$. In other words, if $\Lambda_0 \subset \Lambda$ has hyperbolic product structure generated by the families of stable and unstable discs given by $\Gamma_0^s$ and $\Gamma_0^u$, then $\Lambda_0$ is an $s$-subset if $\Gamma_0^s \subseteq \Gamma^s$ and $\Gamma_0^u = \Gamma^u$. A \emph{u-subset} is defined analogously. 

\begin{definition}\label{Young tower def}
	A $C^{1+\alpha}$ diffeomorphism $f : M \to M$, with $M$ a compact Riemannian manifold, is a \emph{Young's diffeomorphism} if the following conditions are satisfied: 
	\begin{enumerate}[label=(Y\arabic*)]
		\item There exists $\Lambda \subset M$ (called the \emph{base}) with hyperbolic product structure, a countable collection of continuous subfamilies $\Gamma_i^s \subset \Gamma^s$ of stable discs, and positive integers $\tau_i$, $i \in \N$, such that the $s$-subsets
		\[
		\Lambda_i^s := \union_{\gamma \in \Gamma^s_i} \big(\gamma \cap \Lambda \big) \subset \Lambda
		\]
		are pairwise disjoint and satisfy:
		\begin{enumerate}[label=(\alph*)]
			\item \emph{invariance}: for $x \in \Lambda_i^s$, 
			\[
			f^{\tau_i}(\gamma^s(x)) \subset \gamma^s(f^{\tau_i}(x)), \quad \textrm{and} \quad f^{\tau_i}(\gamma^u(x)) \supset \gamma^u(f^{\tau_i}(x)),
			\]
			where $\gamma^{u,s}(x)$ denotes the (un)stable disc containing $x$; and, 
			\item \emph{Markov property}: $\Lambda_i^u := f^{\tau_i}(\Lambda_i^s)$ is a $u$-subset of $\Lambda$ such that for $x \in \Lambda_i^s$, 
			\[
			f^{-\tau_i}(\gamma^s(f^{\tau_i}(x)) \cap \Lambda_i^u) = \gamma^s(x) \cap \Lambda, \quad \textrm{and} \quad f^{\tau_i} (\gamma^u(x) \cap \Lambda_i^s) = \gamma^u(f^{\tau_i}(x)) \cap \Lambda. 
			\]
		\end{enumerate}
		\item For $\gamma^u \in \Gamma^u$, we have
		\[
		\mu_{\gamma^u}(\gamma^u \cap \Lambda) > 0, \quad \textrm{and} \quad \mu_{\gamma^u}\Big( \cl\big( \left(\Lambda \setminus \textstyle\union_i \Lambda_i^s\right) \cap \gamma^u\big)\Big) = 0,
		\]
		where $\mu_{\gamma^u}$ is the induced Riemannian leaf volume on $\gamma^u$ and $\cl(A)$ denotes the closure of $A$ in $M$ for $A \subseteq M$. 
		\item There is $a \in (0,1)$ so that for any $i \in \N$, we have:
		\begin{enumerate}[label=(\alph*)]
			\item For $x \in \Lambda_i^s$ and $y \in \gamma^s(x)$, 
			\[
			d(F(x), F(y)) \leq ad(x,y);
			\]
			\item For $x \in \Lambda_i^s$ and $y \in \gamma^u(x) \cap \Lambda_i^s$, 
			\[
			d(x,y) \leq ad(F(x), F(y)),
			\]
		\end{enumerate}
		where $F : \union_i \Lambda_{i}^s \to \Lambda$ is the \emph{induced map} defined by 
		\[
		F|_{\Lambda^s_i} := f^{\tau_i}|_{\Lambda^s_i}.
		\]
		\item Denote $J^u F(x) = \det\big|DF|_{E^u(x)}\big|$. There exist $c > 0$ and $\kappa \in (0,1)$ such that: 
		\begin{enumerate}[label=(\alph*)]
			\item For all $n \geq 0$, $x \in F^{-n}\left(\union_i \Lambda_i^s\right)$ and $y \in \gamma^s(x)$, we have 
			\[
			\left| \log \frac{J^u F(F^n(x))}{J^u F(F^n(y))}\right| \leq c\kappa^n;
			\]
			\item For any $i_0, \ldots, i_n \in \N$ with $F^k(x), F^k(y) \in \Lambda^s_{i_k}$ for $0 \leq k \leq n$ and $y \in \gamma^u(x)$, we have 
			\[
			\left| \log\frac{J^u F(F^{n-k}(x))}{J^u F(F^{n-k}(y))}\right| \leq c\kappa^k.
			\]
		\end{enumerate}
		\item There is some $\gamma^u \in \Gamma^u$ such that 
		\[
		\sum_{i=1}^\infty \tau_i \mu_{\gamma^u} \left(\Lambda_i^s\right) < \infty. 
		\]
	\end{enumerate}
\end{definition}

\subsection{Realizing $g$ as a Young diffeomorphism}\label{subsec:Realizing $g$ as a Young diffeomorphism}

In Section 7 of \cite{VecPAD1}, it is shown that the smooth nonuniformly hyperbolic pseudo-Anosov diffeomorphism $g : M \to M$ is a Young diffeomorphism. We briefly outline the argument here. 

The first step is to show that the (uniformly hyperbolic) pseudo-Anosov homeomorphism $f : M \to M$ admits a subset $\tilde\Lambda\subset M$ for which Conditions (Y1) - (Y5) are satisfied. To construct $\tilde\Lambda$, we use a finite Markov partition $\tilde \Ps$ for the pseudo-Anosov homeomorphism $f$ (Proposition \ref{pseudo-Anosov Markov}). Note that if $\tilde R \in \tilde \Ps$ is a Markov rectangle, then no singularity of $f$ may lie inside the interior of $\tilde R$ (intuitively this is because $f$ does not admit local hyperbolic product structure at the singularities). Thus, we may take our Markov partition $\tilde \Ps$ of $f$ to be so that the singularities lie on the vertices of the rectangles. 


Recall that $S = \{x_1, \ldots, x_m\}$ denotes the set of singularities, each of which has $p(x_k) = p(k)$ prongs ($1 \leq k \leq m$). Since each singularity $x_k$ is the vertex of a Markov rectangle, there are $2p(k)$ Markov rectangles with $x_k$ as a vertex; we denote these rectangles by $\tilde R_{k,l}$ for $1 \leq k \leq m$, $1 \leq l \leq 2p(k)$. By choosing the diameter of the partition elements to be sufficiently small, we may assume that $\tilde R_{k_1, l_1} \cap \tilde R_{k_2, l_2} = \emptyset$ whenever $k_1 \neq k_2$. 

Let $\tilde R \in \tilde\Ps$ be a partition element that does not intersect the set $\Us_0$ defined in the slow-down procedure for the map $g$. For $x \in \tilde R$, let $\tilde\gamma^s(x)$ and $\tilde\gamma^u(x)$ respectively be the connected components of the stable and unstable leaves through $x$ intersecting $\tilde R$. We call these the \emph{full-length} stable and unstable curves through $x$.

Let $\tilde\tau(x)$ be the first return time of $x$ to $\Int\tilde R$ under $f$ for $x \in \tilde R$. For all $x$ with $\tilde \tau(x) < \infty$, define the set:
\[
\tilde\Lambda^s(x) = \union_{y \in \tilde U^u(x) \setminus \tilde A^u(x)} \tilde\gamma^s(y),
\]
where $\tilde U^u(x) \subset \tilde\gamma^u(x)$ is an interval containing $x$ and open in the induced topology of $\tilde\gamma^u(x)$, and $\tilde A^u(x) \subset \tilde U^u(x)$ is the set of points either lying on the boundary of the Markov partition or never return to the set $\tilde P$. Observe that $\tilde A^u(x)$ has one-dimensional Lebesgue measure equal to 0. One can choose the intervals $\tilde U^u(x)$ so that
\begin{enumerate}[label=(\arabic*)]
	\item for any $y \in \tilde \Lambda^s(x)$, we have $\tilde\tau(y) = \tilde\tau(x)$; and
	\item for any $y \in \tilde R$ with $\tilde\tau(y) < \infty$, there is an $x \in \tilde R$ for which $y \in \tilde\Lambda^s(x)$ and $\tilde\tau(y) = \tilde\tau(x)$. 
\end{enumerate}
Moreover, the image of $\tilde\Lambda^s(x)$ under $f^{\tilde\tau(x)}$ is a $u$-subset containing $f^{\tilde\tau(x)}(x)$. Note that conditions (1) and (2) above ensure that for $x,y \in \tilde R$ with finite return times, the sets $\tilde\Lambda^s(x)$ and $\tilde\Lambda^s(y)$ either coincide or are disjoint. Thus we have a countable collection of disjoint sets $\tilde\Lambda^s_i$ and numbers $\tilde\tau_i$ that give a representation of the pseudo-Anosov homeomorphism $f$ as a Young diffeomorphism with tower base
\[
\tilde\Lambda = \union_{i\geq 1} \tilde\Lambda_i^s.
\]
The sets $\tilde\Lambda_i^s$ form the $s$-sets, $\tilde\Lambda_i^u = f^{\tilde\tau_i}(\tilde\Lambda^s_i)$ form the $u$-sets, and the numbers $\tilde\tau_i$ form the inducing times. See Theorem 7.1 in \cite{VecPAD1} for details. 

Let $H : M \to M$ denote the conjugacy map between $f$ and $g$, so that $g = H\circ f\circ H^{-1}$. Applying $H$ to the Markov partition $\tilde\Ps$, one obtains a Markov partition $\Ps = H(\tilde\Ps)$ of the pseudo-Anosov diffeomorphism $g$. By continuity of $H$, one can construct a Markov partition of $g$ in this way of arbitrarily small diameter. Let $R = H(\tilde R)$, $\Lambda = H(\tilde\Lambda)$. Observe that $\Lambda$ has local hyperbolic product structure given by the full-length stable leaves $\gamma^s(x) = H(\tilde\gamma^s(H^{-1}(x)))$ and the full-length unstable leaves $\gamma^u(x) = H(\tilde\gamma^u(H^{-1}(x)))$. Accordingly, it is shown in \cite{VecPAD1} that $g$ is represented as a Young diffeomorphism with inducing times $\tau_i = \tilde\tau_i$, $s$-sets $\Lambda^s_i = H(\tilde\Lambda^s_i)$, and $u$-sets $\Lambda^u_i = H(\tilde\Lambda^u_i) = g^{\tau_i}(\Lambda^s_i)$. Similarly to the homeomorphism $f$, the inducing times $\tau_i$ are first-return times to $\Lambda$ for points $x \in \Lambda^s_i$ under the $g$. Furthermore, note that if $x \in \Lambda_i^s$, the stable subset $\Lambda_i^s$ satisfies
\[
\Lambda^s_i = \Lambda^s(x) = \union_{y \in U^u(x) \setminus A^u(x)} \gamma^s(y),
\]
where $U^u(x) = H(\tilde U^u(x)) \subset \gamma^u(x)$ is an interval containing $x$ and open in the induced topology of $\gamma^u(x)$, and $A^u(x) = H(\tilde A^u(x)) \subset U^u(x)$ is the set of points that either lie on the boundary of the Markov partition $\Ps$ or never return to $R$. Observe that $A^u(x)$ has one-dimensional Lebesgue measure equal to 0 in $\gamma^u(x)$. 

\begin{proposition}\label{forall-Q}
	Given $Q > 0$, one can choose a Markov partition $\Ps$ for $g$ and the number $r_0$ in the construction of $g$ so that 
	\begin{enumerate}[label=(\arabic*)]
		\item $g^j(x) \not\in \Us_{0}$ for any $0 \leq j \leq Q$ and for any point $x \in M$ for which either $x \in \Lambda$ or $x \not\in \Us_{0}$, while $g^{-1}(x) \in \Us_{0}$; and, 
		\item if $R_{k,l} = H(\tilde R_{k,l})$, with $R_{k,l}$ a Markov rectangle with the singularity $x_k$ as a vertex ($1 \leq k \leq m$, $1 \leq l \leq 2p(k)$), then
		\begin{equation}\label{eq:sing-nbhd-inside-rectangles}
			\Us_{0} \subset \mathrm{int} \union_{k=1}^m \union_{l=1}^{2p(k)} R_{k,l}.
		\end{equation} 
	\end{enumerate} 
\end{proposition}

To prove this proposition, simply note that it holds for the pseudo-Anosov homeomorphism $f$. Applying the conjugacy $H$ yields the result. 

\begin{proposition}[\cite{VecPAD1}]\label{exists-Q}
	There is a $Q > 0$ such that the collection of $s$-sets $\Lambda^s_i$ satisfies Conditions (Y1) - (Y5), thus representing $g : M \to M$ as a Young diffeomorphism. 
\end{proposition}

\section{Behavior near singularities}

In this section, we consider specifically the behavior of trajectories of the system of differential equations given by \eqref{eq:slowdown vector field} in $(s_1, s_2)$-coordinates. The computations in this section pertain specifically to this system of ODEs, and have no \emph{a priori} relation to the manifold $M$, the pseudo-Anosov map $f$, or its smooth realization $g$. 

\begin{remark} 
	Many of the results on the behavior of this system of ODEs that we cite in this section are proven in \cite{PSZ17} and \cite{PSS}. In \cite{PSZ17,PSS}, they use a slowdown function $\psi : [0,1] \to \R$ for which there is an $0 < r_0 < 1$ such that for $u < (r_0/2)^2$, 
\[
\psi(u) = \left( \frac u{r_0}\right)^{\alpha}.
\]
On the other hand, the slowdown function $\Psi_p : [0,1] \to \R$ that we use has constants $0 < r_1 < r_0 < 1$ for which for $u < r_1^2$, we have
\[
\Psi_p(u) = \left( \frac p 2 \right)^{2\alpha} u^\alpha.
\]
In other words, the coefficient $r_0^{-\alpha}$ has been replaced with the coefficient $(p/2)^{2\alpha}$. For this reason, up to a constant multiple, the system of differential equations \eqref{eq:slowdown vector field} is the same as the respective system of differential equations in \cite{PSZ17,PSS}. Accordingly, the results we cite here are proven in \cite{PSZ17,PSS}, up to a multiplicative constant. Several proofs are omitted in this section in the intersest of brevity, but references are given for the respective results in \cite{PSZ17,PSS}.
\end{remark} 




Our next several lemmas concern the trajectories of solutions to equation (\ref{eq:slowdown vector field}). Let $s(t) = (s_1(t), s_2(t))$ be a solution to (\ref{eq:slowdown vector field}). Assume $s(t)$ is defined in the maximal interval $[0,T]$, for which $s(0), s(T) \in \del D_{r_1}$ and $s(t) \in D_{ r_1}$ for $0 < t < T$. Further let $T_1 = T/2$. Note $s_1(t) \leq s_2(t)$ for $0 \leq t \leq T_1$ and $s_1(t) \geq s_2(t)$ for $T_1 \leq t \leq T$. We collect lower and upper bounds on the functions $s_1(t)$ and $s_2(t)$. 

\begin{lemma}\label{s1 and s2 approx}
	Given a solution $s(t)$ to \emph{(\ref{eq:slowdown vector field})}, and $T$ and $T_1$ defined above, we have the following estimates:
	\begin{enumerate}[label=\emph{(\alph*)}]
		\item $|s_2(t)| \geq |s_2(a)| \left( 1 + 2^\alpha C_0 s_2^{2\alpha}(a)(t-a)\right)^{-1/2\alpha}$, $\quad 0 \leq a \leq t \leq T_1$; 
		\item $|s_2(t)| \leq |s_2(a)| \left( 1 + C_0 s_2^{2\alpha}(a)(t-a)\right)^{-1/2\alpha}$, $\quad 0 \leq a \leq t \leq T$;
		\item $|s_1(t)| \geq |s_1(b)| \left( 1 + 2^\alpha C_0 s_1^{2\alpha}(b)(b-t)\right)^{-1/2\alpha}$, $\quad T_1 \leq t \leq b \leq T$;
		\item $|s_1(t)| \leq |s_1(b)| \left( 1 + C_0 s_1^{2\alpha}(b)(b-t)\right)^{-1/2\alpha}$, $\quad 0 \leq t \leq b \leq T$;
	\end{enumerate} 
	where $C_0 = 2\alpha\log\lambda(p/2)^{2\alpha}.$
\end{lemma}


\begin{proof}
	Assume $s_1(t), s_2(t) > 0$ for all $0 \leq t \leq T$. Equation \ref{eq:slowdown vector field} with $\Psi_p(u) = (p/2)^{2\alpha} u^\alpha$ for $0 \leq \alpha \leq r_1^2$ gives us, for $0 \leq t \leq T$ and $i = 1,2$, 
	\begin{equation}\label{eq:PSZlem5.4i}
		\frac{ds_i}{dt} = (-1)^{i+1}\log\lambda \left( \frac p 2\right)^{2\alpha} s_i \left(s_1^2 + s_2^2\right)^{\alpha}
	\end{equation}
	Since $s_i^2 \leq s_1^2 + s_2^2$, we have 
	\[
	\frac{ds_1}{dt} \geq \log\lambda \left( \frac p 2 \right)^{2\alpha} s_1^{2\alpha+1} \quad \textrm{and} \quad \frac{ds_2}{dt} \leq -\log\lambda\left( \frac p 2 \right)^{2\alpha} s_2^{2\alpha+1}.
	\]
	In particular, this implies
	\begin{equation}\label{eq:PSZlem5.4ii}
		s_1(t)^{-2\alpha-1} \frac{ds_1(t)}{dt} \geq \log\lambda\left(\frac p 2 \right)^{2\alpha} \quad \textrm{and} \quad s_2(t)^{-2\alpha-1} \frac{ds_2(t)}{dt} \leq -\log\lambda\left(\frac p 2 \right)^{2\alpha} 
	\end{equation}
	Integrating the inequalities in \eqref{eq:PSZlem5.4ii} over the interval $[a,b] \subset [0,T]$ yields
	\[
		-\frac{1}{2\alpha} \left( s_1(b)^{-2\alpha} - s_1(a)^{-2\alpha}\right) \geq \log\lambda\left( \frac p 2 \right)^{2\alpha} (b-a)
		\]
	and
	\[
	-\frac{1}{2\alpha} \left( s_2(b)^{-2\alpha} - s_2(a)^{-2\alpha}\right) \leq -\log\lambda\left( \frac p 2 \right)^{2\alpha} (b-a),
	\]
	or in other words, 
	\[
	s_1(b)^{-2\alpha} - s_1(a)^{-2\alpha} \leq - C_0(b-a) \quad \textrm{and} \quad s_2(b)^{-2\alpha} - s_2(\alpha)^{-2\alpha} \geq C_0(b-a).
	\]
	Inequalities (b) and (d) all follow by setting $t=a$ or $t=b$. 
	
	Now, for $0 \leq t \leq T_1$, we have $s_2(t) \geq s_1(t)$, and for $T_1 \leq t \leq T$, we have $s_1(t) \geq s_2(t)$. So, 
	\[
	s_1^2 + s_2^2 \leq 2s_2^2 \quad \textrm{for } 0 \leq t \leq T_1
	\]
	and
	\[
	s_1^2 + s_2^2 \leq 2s_1^2 \quad \textrm{for } t_1 \leq t \leq T.
	\]
	It follows from \eqref{eq:PSZlem5.4i} that
	\[
	\frac{ds_2}{dt} \geq -2^\alpha\log\lambda\left(\frac p 2 \right)^{2\alpha} s_2^{2\alpha+1} \quad \textrm{for } 0 \leq t \leq T_1
	\]
	and
	\[
	\frac{ds_1}{dt} \leq 2^\alpha\log\lambda\left( \frac p 2 \right)^{2\alpha} s_1^{2\alpha+1} \quad \textrm{for } T_1 \leq t \leq T.
	\]
	Inequalities (a) and (c) can now be proven in a similar way to inequalities (b) and (d). 
\end{proof}

Consider another solution $\tilde s(t)$ of Equation (\ref{eq:slowdown vector field}) for which $s(0)$ and $\tilde s(0)$ lie in the same quadrant. Set $\Delta s(t) = \tilde s(t) - s(t)$ and $\Delta s_j(t) = \tilde s_j(t) - s_j(t)$, $j=1,2$. 



\begin{lemma}[\cite{PSZ17}, Lemma 5.3 and erratum]\label{spread-in-slowdown lemma}
	Suppose $s_1(t) \neq 0 \neq s_2(t)$ for $t \in [0,T]$ and that $|\tilde s_2(t)| > | s_2(t)|$ for $t \in [0,T]$. Suppose further that $0 < \mu < 1$ satisfies
	\begin{enumerate}[label=\emph{(\arabic*)}]
		\item $\Delta s_2(t) > 0$ and $\left| \Delta s_1(t)\right| \leq \mu |\Delta s_2(t)|$ for $t \in [0,T]$; 
		\item $\left| \frac{\Delta s_2(0)}{s_2(0)}\right| \leq \frac{1-\mu}{72}$. 
	\end{enumerate}
	Then, 
		\begin{align*}
		\Delta s_2(t) &\leq \left|\frac{\Delta s_2(0)}{s_2(0)}\right| |s_2(t)| \left(1 + 2^{\alpha}C_0|s_2(0)|^{2\alpha}t\right)^{-\beta}, & 0 \leq t \leq T_1, \\
		\Delta s_2(t) &\leq \left|\frac{\Delta s_2(T_1)}{s_1(T_1)}\right| |s_1(t)| \left( \frac{1+2^{\alpha}C_0 |s_1(b)|^{2\alpha}(b-t)}{1+2^{\alpha}C_0 |s_1(b)|^{2\alpha}(b-T_1)} \right)^{\beta}, & T_1 \leq t \leq b \leq T, 
		\end{align*}
	where $\beta = \frac{1-\mu}{2^{\alpha+2}}$,
	and $C_0$ is the constant from Lemma \ref{s1 and s2 approx}. Furthermore, 
	\begin{equation}\label{deviation bound}
	\norm{\Delta s(T)} \leq \sqrt{1+\mu^2} \left| \frac{s_1(T)}{s_2(0)}\right| \norm{\Delta s(0)}. 
	\end{equation}
\end{lemma}

Given an exponent $0 < \alpha < 1$ and a parameter $0 < \mu < 1$ as in Lemma \ref{spread-in-slowdown lemma}, define
\begin{equation}\label{gamma-defs}
\gamma = \frac{1}{2\alpha} + 2^{\alpha-1}(1+\mu) + \frac{1-\mu}{6} \quad \textrm{and} \quad \gamma' = \frac{1}{2\alpha} + \frac{1-\mu}{2^{\alpha+2}}. 
\end{equation}
Note $\gamma > \gamma' > 2$ for $0 < \alpha < 1/4$ and $0 < \mu < 1/2$. 

\begin{lemma}[\cite{PSS}, Lemma 6.4]\label{lem:PSSlem6.4}
Under the assumptions of Lemma \ref{spread-in-slowdown lemma}, there is a $C_1 > 0$ for which for any $0 \leq t \leq T_1$, 
\[
|\Delta s_2(t)| \leq C_1|\Delta s_2(0)|t^{-\gamma'}.
\]
\end{lemma}

\color{black}

\begin{lemma}[\cite{PSS}, Lemma 6.5]\label{lem:PSS-6.5}
	Under the assumptions of Lemma \ref{spread-in-slowdown lemma}, one has 
	\begin{align*}
	\Delta s_2(t) &\geq \left|\frac{\Delta s_2(0)}{s_2(0)}\right| |s_2(t)| \left(1 + C_0 |s_2(0)|^{2\alpha}t\right)^{-\beta_1}, &0 \leq t \leq T_1; \\
	\Delta s_2(t) &\geq \left|\frac{\Delta s_2(T_1)}{s_1(T_1)}\right| |s_1(t)| \left(1+C_0 |s_1(T_1)|^{2\alpha}(t-T_1)\right)^{-\beta_2}, & T_1 \leq t \leq T,
	\end{align*}
	where 
	\[
	\beta_1 = (1+\mu)2^{\alpha-1} + \frac{1-\mu}6 \quad \textrm{and} \quad \beta_2 = \beta_1 + \frac{2\alpha}{\alpha}.
	\]
\end{lemma}

\begin{lemma}\label{lem:PSS Lemma 6.6}
	Under the assumptions of Lemma \ref{spread-in-slowdown lemma}, there exists a $C_2 > 0$ for which for any $0 \leq t \leq T_1$,
	\begin{align*}
	|\Delta s_2(t)| &\geq C_2 |\Delta s_2(0)|,  &0 < t < 1, \\
	|\Delta s_2(t)| &\geq C_2 |\Delta s_2(0)| t^{-\gamma}, & t \geq 1.
	\end{align*}
\end{lemma}

\begin{proof}
By inequality (a) in Lemma \ref{s1 and s2 approx} and the first inequality in Lemma \ref{lem:PSS-6.5}, for $0 < t < T_1$, we have: 
\begin{align*}
\Delta s_2(t) &\geq \left|\frac{\Delta s_2(0)}{s_2(0)}\right| |s_2(t)|\left( 1 + C_0 |s_2(0)|^{2\alpha}t\right)^{-\beta_1}\\
&\geq | \Delta s_2(0)| \left( 1+2^{\alpha}C_0 |s_2(0)|^{2\alpha}t\right)^{-\beta_1 - 1/2\alpha}.
\end{align*}
For $0 < t < 1$, since $|s_2(0)| \leq r_1$, we're done by setting $$C_2 = \left( 1 + 2^{(p-2)/p} C_0 r_1^{(2p-4)/p} \right)^{-\beta_1 - 1/2\alpha}.$$ For $t \geq 1$, since $1+At \leq (1+A)t$ for $A > 0$, we have 
\begin{align*}
|\Delta s_2(t)| &\geq | \Delta s_2(0)| \left( 1+2^{\alpha}C_0 |s_2(0)|^{2\alpha}\right)^{-\beta_1 - 1/2\alpha} t^{-\beta_1 - 1/2\alpha}.
\end{align*}
Noting $\gamma = \beta_1 + \frac 1{2\alpha}$, the same $C_2$ as in the $t < 1$ case satisfies the second estimate in the lemma. 
\end{proof}

\begin{lemma}[\cite{PSS}, Lemma 6.7]\label{lem:PSS Lemma 6.7}
Under the assumptions of Lemma \ref{spread-in-slowdown lemma}, there exist $C_3, C_4 > 0$ such that 
\[
C_3 \Delta s_2(T_1) \geq \Delta s_2(T) \geq C_4\Delta s_2(T_1).
\]
\end{lemma}


\section{A lower bound on the tail of the return time}

Proving Theorem \ref{Main 1} requires polynomial upper and lower bounds on the tail of the return time, $\mu_1\left(\left\{x \in \Lambda : \tau(x) > n\right\}\right)$ (where $\mu_1$ is the $g$-invariant Riemannian measure from Proposition \ref{prop:SRB}). We prove these bounds in this section. 

To begin, we cite the following result, bounding the time a typical orbit stays near a singularity. 

\begin{lemma}[\cite{VecPAD1}, Lemma 5.2]\label{annular bound in M}
	There exists a $T_0 \in \Z$, depending on $r_0$ and $\lambda$, so that for any $x \in \Us_0 = \union_{k=1}^m \oldphi^{-1}_k(D_{\rho_0})$, we have 
	\[
	\max\left\{ N > 0 : g^n(x) \in \union_{k=1}^m \oldphi_k^{-1}\left(D_{\rho_0} \setminus D_{\rho_1}\right) \: \textrm{for all } n = 0, \ldots N\right\} \leq T_0.
	\]
\end{lemma}

Now, consider the Young structure on $(M,g)$ constructed in Section 5 with stable sets $\Lambda_s^i$. Note the sets $\Lambda^s_i$ consist of full-length stable curves through a Markov rectangle $R$. Fix one such curve $\sigma$. Denote $D^k_{\rho_j} = \oldphi_k^{-1}\left(D_{\rho_j}\right) \subset M $ for $j=0,1$. 

\begin{lemma}\label{lem:length ratio is polynomial}
	Suppose a stable curve $\sigma \subset \Lambda_i^s$ enters a singular neighborhood $D^k_{\rho_1}$ at time $n > 1$, so that $g^n(\sigma) \cap D^k_{\rho_1} \neq \emptyset$, and that $\sigma$ exits $D^k_{\rho_1}$ at time $m > n$. Then, 
	\begin{equation}\label{eq:length ratio is polynomial}
	C_5(m-n)^{-\gamma} \leq \frac{L\left(g^m(\sigma)\right)}{L\left(g^n(\sigma)\right)} \leq C_6(m-n)^{-\gamma'},
	\end{equation}
	where $C_5>0$, $C_6 > 0$ are constants independent of $m$, $n$, and the choice of stable curve $\sigma$; $\gamma, \gamma'$ are as in (\ref{gamma-defs}); and $L$ denotes the length of the curve. 
\end{lemma}

\begin{proof}
	Let $x, y$ be the endpoints of the curve $\gamma$ in $R$. Set $x_k = g^k(x)$ and $y_k = g^k(y)$. Observe that there is a $K_0 > 0$ such that for all $k \geq 1$, 
	\begin{equation}\label{eq:length-dist-equivalence}
	K_0^{-1} d(x_k, y_k) \leq L(g^k(\sigma)) \leq K_0 d(x_k, y_k).
	\end{equation}
	where $d$ is the Riemannian distance in $M$. 
	
	Let $\sigma$ be as in the statement of the lemma. By assumptions on the pseudo-Anosov homeomorphism $f$, $\sigma$ remains in a stable sector for the duration of time it remains in $D^k_{\rho_1}$ prior to exiting. In $(s_1, s_2)$-coordinates in the stable sector, it is enough to consider the map $g : M \to M$ near $x_k$ to be the time-1 map $G_p : \R^2 \to \R^2$ of the vector field \eqref{eq:slowdown vector field}. (Recall $s_1 + is_2 = \left( \varphi^{-1} \circ \Phi_{kj} \circ \phi_k\right)(x)$ for $x \in \Us_0$; see \eqref{eq:mass-push-coord-change}, \eqref{eq:sector-coord-map}, and \eqref{eq:g-def}.)
	
	Let $s, \tilde s : [0,m-n] \to \R^2$ be solutions to \eqref{eq:slowdown vector field} with initial conditions $s(0) = \left(\phi^{-1} \circ \Phi_{kj} \circ \oldphi_k\right)(x_n)$ and $\tilde s(0) = \left(\phi^{-1} \circ \Phi_{kj} \circ \oldphi_k\right)(y_n)$, and note $s(0), \tilde s(0) \in D_{r_1}$ while $s(m-n),\tilde s(m-n)$ lies in $D_{r_0} \setminus D_{r_1}$. Also define $\Delta s_i(t) = \tilde s_i(t)  - s_i(t)$ and let $\Delta s(t) = (\Delta s_1, \Delta s_2) \in \R^2$ be the difference vector from $\tilde s(t)$ to $s(t)$. Note that there is a $K_1$ independent of $\sigma$ for which 
	\begin{equation}\label{eq:flow-length-estimate}
	K_1^{-1} \|\Delta s(j-n)\| \leq d(x_j, y_j) \leq K_1\|\Delta s(j-n)\|
	\end{equation}
	for all $n \leq j \leq m$.
	
	We will apply Lemma \ref{spread-in-slowdown lemma}. To check that the conditions are satisfied, first observe that Assumption 1 is satisfied since $y$ is in the image of the stable cone of $x$ under $\exp_x : T_x M \to M$. Assumption 2 is satisfied if $d(x_k, y_k)$, for $k=n,m$, is sufficiently small. This can be done, using Proposition \ref{forall-Q} and \eqref{eq:flow-length-estimate}, by taking $r_0 > 0$ in the construction of $g$ so that $Q>0$ is sufficiently large. So Lemma \ref{spread-in-slowdown lemma} applies.
	
	Assume $\|\Delta s(0)\|$ is made sufficiently small so that the curve $G_p^j\left( \phi^{-1} \circ \Phi_{kj} \circ \oldphi_k(\sigma)\right) \in \R^2$ lies in $D_{r_0/2} \cap \{(s_1, s_2) : s_1 > s_2\}$ for $n < j < \frac{n+m}{2}$ and lies in $D_{r_0/2} \cap \{(s_1, s_2) : s_1 < s_2\}$ for $\frac{n+m}{2} < j < m$. Applying Lemmas \ref{spread-in-slowdown lemma}, \ref{lem:PSS Lemma 6.6}, and \ref{lem:PSS Lemma 6.7} (with $T = m-n$ and $T_1 = (m-n)/2$), as well as \eqref{eq:length-dist-equivalence} and \eqref{eq:flow-length-estimate}, we obtain: 
	\begin{align*}
		L(g^m(\sigma)) &\geq K_0^{-1} d(x_m, y_m) \\
		&\geq K_0^{-1} K_1^{-1} \|\Delta s(m-n)\| \\
		&\geq K_0^{-1} K_1^{-1} |\Delta s_2(m-n)| \\
		&\geq K_0^{-1} K_1^{-1} C_4\left| \Delta s_2\left( \frac{m-n}2\right)\right| \\
		&\geq K_0^{-1} K_1^{-1} C_2 C_4|\Delta s_2(0)| \left( \frac{m-n}2\right)^{-\gamma} \\
		&\geq K_0^{-1} K_1^{-1} C_2 C_4 2^\gamma (m-n)^{-\gamma} \frac{1}{\sqrt{1+\mu^2}} \|\Delta s(0)\| \\
		&\geq K_0^{-2} K_1^{-2} C_2C_4 2^{\gamma} (m-n)^{-\gamma} \frac{1}{\sqrt{1+\mu^2}} L(g^n(\sigma)).
	\end{align*}
	The lower bound of \eqref{eq:length ratio is polynomial} now follows with $C_5 = K_0^{-2} K_1^{-2} C_2 C_4 2^\gamma/\sqrt{1+\mu^2}$.
	
	To prove the upper bound, we use Lemmas \ref{spread-in-slowdown lemma}, \ref{lem:PSSlem6.4}, and \ref{lem:PSS Lemma 6.7}, as well as \eqref{eq:length-dist-equivalence} and \eqref{eq:flow-length-estimate}, to show:
	\begin{align*}
		L(g^m(\sigma)) &\leq K_0 d(x_m, y_m) \\
		&\leq K_0 K_1 \|\Delta s(m-n)\| \\
		&\leq K_0 K_1 \sqrt{1+\mu} |\Delta s_2(m-n)| \\
		&\leq K_0 K_1 C_3 \sqrt{1+\mu}\left| \Delta s_2\left(\frac{m-n}2\right)\right| \\
		&\leq K_0 K_1 C_1 C_3 2^{\gamma'}\sqrt{1+\mu}|\Delta s_2(0)|(m-n)^{-\gamma'} \\
		&\leq K_0 K_1 C_1 C_3 2^{\gamma'} \sqrt{1+\mu} (m-n)^{-\gamma'} \|\Delta s(0)\| \\
		&\leq K_0^2 K_1^2 C_1 C_3 2^{\gamma'} \sqrt{1+\mu} (m-n)^{-\gamma'} L(g^n(\sigma)).
	\end{align*}
	The upper bound of \eqref{eq:length ratio is polynomial} now follows with $C_6 = K_0^2 K_1^2 C_1 C_3 2^{\gamma'} \sqrt{1+\mu}$.
\end{proof} 

\color{black}

Let $\widetilde\Ps$ and $\Ps$ be Markov partitions for the pseudo-Anosov homemorphism $f$ and the smooth realization $g$ respectively, and let $\tilde R \in \tilde\Ps$ and $R \in \Ps$ be the partition elements discussed in Section \ref{subsec:Realizing $g$ as a Young diffeomorphism}. Fix the number $Q$ as in Proposition \ref{exists-Q}. Assume the partition $\Ps$ and the numbers $0 < r_1 < r_0$ are chosen so that Proposition \ref{forall-Q} holds. Finally, denote:
\[
\Ns = \tau(\Lambda) = \{n \in \N : \textrm{there exists } x \in R \textrm{ such that } n = \tau(x)\}.
\]
\begin{lemma}\label{time bound near singularities}
	We may choose $\rho_0>0$ in the construction of $g$ so that there is an integer $Q_0 > 0$ satisfying the following property: For each singularity $x_k$, for any $N > 0$, one can find $n \in \Ns$ with $n > N$, an $s$-subset $\Lambda^s_l$ with $\tau(\Lambda^s_l) = n$ and numbers $0 < m_1 < m_2$ satisfying $m_1 < Q_0$, $n-m_2 < Q_0$ such that $g^l(\Lambda^s_l) \cap \Us_0 = \emptyset$ for all $0 \leq l < m_1$ and $m_2 < l \leq n$, and $g^l(\Lambda_l^s) \cap \Us_{0} \neq \emptyset$ for all $m_1 \leq l \leq m_2$. 
\end{lemma}

\begin{proof}
	We will show that for each $k = 1, \ldots, \ell$ (where $\ell$ is the number of singularities of $f$), there is an integer $Q_k > 0$ satisfying this proposition with $\Us_{0}$ replaced with the neighborhood $D^k_{\rho_0}$ around the singularity $x_k$. Taking $Q_0 = \max\{Q_1, \ldots, Q_m\}$ will yield the result. 
	
	Fix $k \in \{1, \ldots, m\}$. To prove the existence of $Q_k$, it suffices to show there exists an integer $Q_k > 0$ such that for any $N > 0$, there is an admissible word of length $n > N$ of the form 
	\begin{equation}\label{admissible word}
	R\oldbar W_1 \oldbar R_k \oldbar W_2 R,
	\end{equation}
	where the words $\oldbar W_1$ and $\oldbar W_2$ are of length $\left|\oldbar W_q\right| < Q_k$ for $q=1,2$, and do not contain any of the symbols $R$ or $R_{j,l}$ (the latter being elements of the Markov partition with a singularity $x_j$ as a vertex; see Proposition \ref{forall-Q}), and the word $\oldbar R_k$ consists of one of the symbols $R_{k,1}, \ldots, R_{k, 2p(k)}$ repeated $\left|\oldbar R_k\right| = n-2-\left|\oldbar W_1\right| - \left|\oldbar W_2\right|$ times (since the stable and unstable sectors satisfy $g(S^{s/u}_{kj}) \subset S^{s/u}_{kj}$; see section \ref{sec:construction}). Observe that since this word of length $n$ begins and ends with the symbol $R$, we have that $n \in \Ns$. 
	
	Because the smooth realization $g$ is topologically conjugate to the linear pseudo-Anosov homeomorphism $f$, it suffices to prove that there is an admissible word of the form (\ref{admissible word}) for $f$ and the Markov partition $\tilde\Ps$. To this end, consider the stable and unstable prongs through $x_k$. Since the singularities are fixed points by assumption, the prongs are invariant under $f$. As before, let $P^s_{k,j} \subset D^k_{\rho_0}$ and $P^u_{k,j} \subset D^k_{\rho_0}$ ($1 \leq j \leq p(k)$) be the components of the stable and unstable prongs having $x_k$ as an endpoint and contained in $D^k_{\rho_0}$. By topological transitivity of $f$ (\cite{FS79}, Corollary 9.19), we know $f^q(\tilde R) \cap D^k_{\rho_0} \neq \emptyset$ for some integer $q \geq 1$ (recall $\tilde R$ is the Markov element of $f$, corresponding to the Markov element $R$ of $g$ under topological conjugacy). For each $j = 1, \ldots, p(k)$, there are minimal positive integers $n_j^s$ and $n_j^u$ for which $f^{-n_j^s}(P^s_{k,j}) \cap \tilde R \neq \emptyset$ and $f^{n^u_j}(P^u_{k,j}) \cap \tilde R \neq \emptyset$. For definiteness, without loss of generality, assume $n^s_1 = \max\left\{n^s_j : 1 \leq j \leq p(k) \right\}$, and let $\gamma^s$ and $\gamma^u$ accordingly be full-length stable and unstable curves in $\tilde R$ for which $P^s_{k,1} \supset f^{n^s_1}(\gamma^s)$ and $P^u_{k,1} \supset f^{-n^u_1}(\gamma^u)$. In particular, $\gamma^s$ and $\gamma^u$ are constructed so that they lie on stable and unstable manifolds that extend from stable and unstable prongs of $x_k$. By reducing $\rho_0$ if necessary (which we would need to do only finitely many times), we may assume $f^i(\gamma^s)$ and $f^{-i}(\gamma^u)$ enter $\Us_0$ for the first and only time when $f^{n^s_1}(\gamma^s) \subset P^s_{k,1}$ and $f^{-n^u_1}(\gamma^u) \subset P^u_{k,1}$. It follows that $f^l(\gamma^s) \cap D^k_{\rho_0} = \emptyset$ for $0 \leq l < n_1^s$ and $f^{-l}(\gamma^u) \cap D^k_{\rho_0} = \emptyset$ for $0 \leq l < n_1^u$. 
	
	Since the manifolds extending the prongs are invariant under $f$, observe that $f^i(\gamma^s)$ and $f^{-i}(\gamma^u)$ never return to $\tilde R$ as $i \to \infty$. Thus, for any $n \in \Ns$ with $n > N$, there is a $u$-subset $\tilde\Lambda^u_{j_1}$ which completely enters $D^k_{\rho_0}$ at the same time as $\gamma^u$ (iterated under $f^{-1})$, and an $s$-subset $\tilde\Lambda^s_{j_2}$ which completely enters $D^k_{\rho_0}$ at the same time as $\gamma^s$ (iterated under $f$). Recall that $\gamma^s \subset \tilde R$ is an extension of the stable prong at the singularity $x_k$. Taking a point $x \in \tilde\Lambda^s_{j_2}$ sufficiently close to $\gamma^s$, we note that eventually $f^i(x) \in f^{-n_1^u}(\Lambda^u_{j_1})$, and so $f^{i+n_1^u}(x) \in \tilde R$. So the symbolic representation of $x$ satisfies \eqref{admissible word}, with $Q_k = \max\{n_1^s, n_1^u\}$. This completes the proof of the lemma. 
\end{proof}

%

\begin{lemma}\label{tail lower bound}
	There exists a constant $C_7 > 0$ such that 
	\[
	\mu_1\left(\left\{x \in \Lambda : \tau(x) > n \right\}\right) > C_7 n^{-(\gamma-1)},
	\]
	where $\mu_1$ is the measure of Proposition \ref{prop:SRB} and $\gamma$ is defined in (\ref{gamma-defs}). 
\end{lemma}

\begin{proof}
	We begin by observing
	\begin{align*}
	\mu_1\left(\left\{x \in \Lambda : \tau(x) > n\right\}\right) &= \sum_{N=n+1}^\infty \mu_1(\{x \in \Lambda : \tau(x) = N\}) \\
	&= \sum_{N = n+1}^\infty \sum_{\Lambda^s_k : \tau(\Lambda^s_k) = N} \mu_1(\Lambda^s_k) \\
	&> \sum_{N=n+1}^\infty \mu_1(\Lambda_l^s(N)),
	\end{align*}
	where $\Lambda_l^s(N) =: \Lambda_l^s$ is the $s$-set defined in Lemma \ref{time bound near singularities}. We will show that there is a $K > 0$ for which 
	\begin{equation}\label{TLB 0}
	\mu_1(\Lambda_l^s(N)) \geq KN^{-\gamma}
	\end{equation}
	for each $N \geq n+1$, where $\gamma > 0$ is given in \eqref{gamma-defs}. Once this is shown, we have
	\[
	\mu_1(\{x \in \Lambda : \tau(x) > n\}) > \sum_{N = n+1}^\infty KN^{-\gamma} > C_7 n^{-(\gamma-1)}
	\] 
	for some constant $C_7>0$. 
	
	Given $x \in \Lambda_l^s$, let $\gamma^s_l(x) = \gamma^s(x) \cap \Lambda^s_l$ (where $\gamma^s(x)$ is the full-length stable leaf through $x$ in the Markov rectangle $R$). Since the $g$-invariant measure $\mu_1$ is determined locally by the product structure of the stable and unstable manifolds (by Lemma \ref{lem:lebesgue-near-sings} and the definition of $\Omega_p$), there is a constant $K_1> 0$ independent of $x \in \Lambda_l^s$ such that 
	\begin{equation}\label{TLB 1}
	\mu_1(\Lambda_l^s) = \mu_1(g^N(\Lambda_l^s)) = K_1 L(g^N(\gamma_l^s(x)))
	\end{equation}
	where $L$ denotes the length of the curve. 
	
	Let $x_j = g^j(x)$ for $j = 0, \ldots, N$. By Lemma \ref{time bound near singularities}, there are $k_1, k_2 \geq 1$ such that $g^j(x) \not\in \Us_{0}$ if $0 \leq j < k_1$ or if $k_2 < j \leq N$, and $g^j(x) \in \Us_{0}$ if $k_1 \leq j \leq k_2$. Note that for $0 \leq j < k_1$ and $k_2 < j \leq N$, the curve $g^j(\gamma_l^s(x))$ lies in the stable cone for the pseudo-Anosov homeomorphism $f$ at $x_j$, and indeed, is an admissible manifold for $f$ (i.e., for $y \in g^j(\gamma_l^s(x))$, the tangent line $T_y g^j(\gamma^s_l(x))$ lies in the stable cone at $y$). Thus, the length of the curve $\gamma^s(x)$ contracts exponentially outside of the region $\Us_{0}$ with contracting constant $\lambda^{-1}$ (where we recall $\lambda > 1$ is the expansion constant for the pseudo-Anosov homeomorphism $f$). By the proof of Lemma \ref{time bound near singularities}, we have that $\gamma^s_l(x)$ enters and exits $\Us_{0}$ at the same time as $\Lambda_l^s$, so $k_1 < Q_0$ and $N-k_2 < Q_0$. Therefore, 
	\begin{equation}\label{TLB 2}
	L(\gamma^s_l(x)) = \lambda^{k_1} L(g^{k_1}(\gamma_l^s(x))) \leq \lambda^{Q_0}L(g^{k_1}(\gamma_l^s(x)))
	\end{equation} 
	and
	\begin{equation}\label{TLB 3}
	L(g^N(\gamma_l^s(x))) = \lambda^{-(N-k_2)}L(g^{k_2}(\gamma_l^s(x))) \geq \lambda^{-Q_0}L(g^{k_2}(\gamma_l^s(x))).
	\end{equation}
	
	Let $\Us_1 = \union_{k=1}^m D^k_{\rho_1}$, where we recall $m$ is the number of singularities and $D^k_{\rho_1} = \phi_k^{-1}(D_{\rho_1})$ is the neighborhood of the singularity $x_k$ given as the preimage of $D_{\rho_1} \subset \C$. By Lemma \ref{annular bound in M}, the time the trajectory spends in $\Us_{0} \setminus \Us_{1}$ is uniformly bounded. So by Lemma \ref{lem:length ratio is polynomial}, there is a constant $\hat C_6 > 0$ such that 
	\begin{equation}\label{TLB 4}
	L(g^{k_2}(\gamma_l^s(x))) > \hat C_6(k_2 - k_1)^{-\gamma} L(g^{k_1}(\gamma_l^s(x))). 
	\end{equation}
	Since $k_2 - k_1 < N$, by (\ref{TLB 1}) - (\ref{TLB 4}), 
	\begin{align*}
	\mu_1(\Lambda_l^s) &\geq K_1 L(g^N(\gamma_l^s(x))) \geq K_1 \lambda^{-Q_0} L(g^{k_2}(\gamma_l^s(x))) \\
	&> K_1  \hat C_6 \lambda^{-Q_0}(k_2 - k_1)^{-\gamma} L(g^{k_1}(\gamma_l^s(x))) \\
	&> K_1  \hat C_6 \lambda^{-2Q_0}L(\gamma_l^s(x))N^{-\gamma}.
	\end{align*}
	Note that since $\gamma_l^s(x)$ is a full-length stable curve in $R$, the length of $\gamma_l^s(x)$ is independent of $N$. So the value $K = K_1 C_6 \lambda^{-2Q_0}L(\gamma_l^s(x))$ is independent of $N$. This proves \eqref{TLB 0}. 
\end{proof}

\section{An upper bound on the tail of the return time} 

We now prove that the tail of the return time of the Young structure of $g$ has a polynomial upper bound. Recall that $R$ is the element of the Markov partition of $g$ containing the base of the Young tower, $\Us_{0}$ is the $\rho_0$-neighborhood of the singularities, and $R\cap \Us_{0} = \emptyset.$ Given an $s$-set $\Lambda_i^s \subset R$ with $\tau(\Lambda_i^s) = n$, choose integers $q = q(\Lambda_i^s)$ and $r = r(\Lambda_i^s)$, and two finite collections of numbers $\{k_j\geq 0\}_{j = 1, \ldots, q}$ and $\{l_j \geq 0\}_{j=0,\ldots, q}$ such that
\begin{enumerate}
	\item $k_1 + k_2 + \cdots + k_q = k \quad \textrm{and} \quad l_0 + l_2 + \ldots + l_{q} = n-k$;
	\item the trajectory of the set $\Lambda_i^s$ under $g^j$, $0 \leq j \leq n$, consecutively spends $l_q$ time outside $\Us_0$ and $k_q$ times inside $\Us_0$.
\end{enumerate}
Consider now the set of $s$-sets
\[
\Ss_{k,n,q} = \{\Lambda_i^s \subset R : \tau(\Lambda_i^s) = n, k = k(\Lambda^s), q = q(\Lambda_s^i)\}.
\]
Thus $\Ss_{k,n,q}$ is the set of $s$-sets with return time $\tau(\Lambda_i^s) = n$ and that spend a total of $k$ time outside of $\Us_0$ before returning to $\Lambda_i^s$, and enter $\Us_0$ in total $q$ times. 

\begin{lemma}\label{lem:PSS8.1}
	There are $0 < h < \htop(g)$, $\epsilon_0 > 0$, and $C_8 > 0$ such that $\epsilon_0 < \htop(g)-h$ and 
	\begin{equation}\label{eq:word-limit}
\#\Ss_{k,n,q} \leq C_8 \frac 1{q^2} e^{(h+\epsilon_0)(n-k)}.
	\end{equation}
\end{lemma} 

\begin{proof}
	Recall that $\tilde\Ps$ is the Markov partition for the pseudo-Anosov homeomorphism $f : M \to M$, and that $H : M \to M$ is the conjugacy map between the pseudo-Anosov homeomorphism $f$ and its smooth model $g$, so $g \circ H = H \circ g$. Further, recall that for each singularity $x_l$, $l = 1, \ldots, m$, the Markov rectangle $\tilde R_{j,p(l)} \in \tilde\Ps$, for $1 \leq j \leq 2p(l)$, is one of the $2p(k)$ rectangles with the singularity $x_k$ as a vertex. Let $R_{j,p(l)} = H(\tilde R_{j,p(l)})$. Define the set $V = \union_{l=1}^m \union_{j=1}^{2p(l)} R_{j,l}$, and the number 
	\[
	s := \sum_{l=1}^m 2p(l)
	\]
	to be the number of Markov rectangles making up $V$, i.e., the number of Markov rectangles with a vertex containing a singularity.

	Observe that for a particular $\Lambda_i^s \in \Ss_{k,n,q}$, the symbolic representation of every $x \in \Lambda_i^s$ has the same first $n = \tau(\Lambda_i^s)$ symbols, which begin and end with $R$. By \eqref{eq:sing-nbhd-inside-rectangles}, it follows that the cardinality of $\Ss_{k,n,q}$ is less than or equal to the set of all words of length $n$ that begin and end with $R$, and which contain $k$ instances of the symbols $R_{j,p(l)}$, $1 \leq l \leq m$, $1 \leq j \leq 2p(l)$, and for which the remaining $n-k$ symbols do not have singularities in their closures. We will show that the number of such words is bounded by \eqref{eq:word-limit}. 
	
	Given $k$ and $q$, the number of ways $k$ can be partitioned into $q$ summands respecting order is $\binom{k-1}{q-1}$, and so the number of ways the orbit of $x \in \Lambda_i^s$ can enter the set $V$ $p$ times with total time in $V$ not exceeding $k$ is $\leq \binom{k-1}{q-1}$. Likewise, the number of ways $n-k$ can be partitioned into $q+1$ summands respecting order is $\binom{n-k-1}{q}$, and so the number of ways the orbit of $x \in \Lambda_i^s$ can enter $V^c$ (counting the ``zeroth'' entry when it starts in $R$) without exceeding $n-k$ is  $\leq \binom{n-k-1}{q}$. So there are $\binom{k-1}{q-1} \binom{n-k-1}{q}$ pairs of ordered sets of integers $(k_1, \ldots, k_q)$, $(l_0, \ldots, l_q)$ for which $k_1 + \cdots + k_q = k$ and $l_0 + \cdots + l_q = n-k$. 
	
	Consider one such pair of ordered sets $(k_1, \ldots, k_q)$, $(l_0, \ldots, l_q)$. By assumption, the map $f$ (and thus the map $g$) preserve the stable sectors $S_{jl}^s$ around each singularity $x_l$. Assume the Markov partition is sufficiently small so that each rectangle is contained in one of the coordinate charts $U_j$ defining the homeomorphism $f$. It follows that when the orbit of $x \in \Lambda_i^s$ enters $V$ the $r^\textrm{th}$ time, its symbolic representation contains $k_r$ copies of a single symbol $R_{j,p(l)}$. Therefore, the first $n$-letter word in the symbolic representation of an $x \in \Lambda_i^s$ with times $(k_1, \ldots, k_q)$ spent in $V$ and times $(l_0, \ldots, l_q)$ spent outside of $V$, is of the form 
	\[
	\underline{R}_{l_0} [R_{j(1),l(1)}]^{k_1} \underline{R}_{l_1} [R_{j(2),l(2)}]^{k_2} \cdots [R_{j(q),l(q)}]^{k_q} \underline{R}_{l_q}
	\]
	where each $\underline{R}_{l_r}$ is a word in $\Ps$ of length $l_r$ not including letters in $V$, and $[R_{j(r),l(r)}]^{k_r}$ is a word made of $k$ copies of $R_{j(r),l(r)}$. Observe that for each $(k_1, \ldots, k_q)$, there are $s^q$ possible configurations of $[R_{j(1),l(1)}]^{k_1}, \ldots, [R_{j(q),l(q)}]^{k_q}$. 
	
	Now consider a word of length $l_r$. Given a topologically mixing topological Markov shift $\sigma : \As^\Z\to \As^\Z$ over an alphabet $\As$, and a set $\Bs \subset \As$ of forbidden letters, there is a $C>0$ and an $h \in (0, \htop(\sigma))$ for which the number of words of length $n$ not including any symbols from $\Bs$ is $\leq Ce^{nh}$. Since the Markov shift associated to $g : M \to M$ with symbols $\Ps$ is topologically mixing (as all pseudo-Anosov homeomorphisms on surfaces are topologically transitive), it follows that the number of words of length $l_r$ not including the symbols in $\{R_{j,l}\}$ is $\leq C_8e^{hl_r}$, where $C_8> 0$ is independent of $l_r$ and $h < \htop(g) = \htop(f)$. So to summarize, for $k,n,q \geq 1$, there are $\binom{k-1}{q-1}\binom{n-k-1}{q}$ possible pairs of ordered sets $(k_1, \ldots, k_q)$, $(l_0, \ldots, l_q)$ with $k_1 + \cdots + k_q = k$, $l_0 + \ldots + l_q = n-k$; for each such ordered set $(k_1, \ldots, k_q)$, there are $s^q$ possible configurations of the $R_{j(r),l(r)}$, $1 \leq r \leq q$; and for each $l_r$, there are $\leq C_8e^{hl_r}$ possible words of length $l_r$. Since $l_0 + \cdots + l_q = n-k$, it follows that 
	\begin{equation}\label{eq:combinatorial-approx}
	\begin{aligned}
	\# \Ss_{k,n,q} &\leq C_8s^q \binom{k-1}{q-1} \binom{n-k-1}{q} e^{h(n-k)} \\
	&= \frac{C_8}{q^2}q^2 s^q \binom{k-1}{q-1} \binom{n-k-1}{q} e^{h(n-k)} .
	\end{aligned} 
	\end{equation}
	Our goal is to estimate the quantity $q^2 s^q \binom{k-1}{q-1} \binom{n-k-1}{q}$. 
	
	We begin by bounding $\binom{k-1}{q-1}$. By Proposition \ref{forall-Q}, it takes $\Lambda_i^s$ at least $Q$ iterates before it reenters $\Us_{\rho_0}$ after exiting (or after starting from the rectangle $R$). This means $n = k + l_0 + \ldots + l_q > k+(q+1)Q$, i.e., $q+1 < \frac{n-k}{Q}.$ Now, for a fixed $k$, $\binom{k-1}{q-1}$ is maximized when $q-1 = \left\lfloor \frac{k-1}{2}\right\rfloor$. It follows that $\left\lfloor \frac{k-1}{2}\right\rfloor < \frac{n-k}Q$. Using the asymptotic formula $\binom{a}{b} < \left(\frac{ae}{b}\right)^b$, we obtain: 
	\begin{equation}\label{eq:(k-1)(q-1)-est}
	\begin{aligned}
	\binom{k-1}{q-1} &\leq \binom{k-1}{\left\lfloor \frac{k-1}{2}\right\rfloor} \\ &< \left( \frac{(k-1)e}{\left\lfloor \frac{k-1}{2}\right\rfloor}\right)^{\lfloor (k-1)/2\rfloor} \\
	&\leq (2e)^{\lfloor(k-1)/2\rfloor} \\
	&< (2e)^{\frac{n-k}Q} \\
	&< e^{\frac{n-k}Q \ln(2e)}.
	\end{aligned}
	\end{equation}
	
	Next we estimate $\binom{n-k-1}{q}$. Note $q < \min\left\{ \frac{k-1}{2}, \frac{n-k}{Q}\right\}$, and so using the asymptotic formula from earlier, we observe: 
	\begin{equation}\label{eq:(n-k-1)(q)-est}
	\begin{aligned}
	\binom{n-k-1}{q} < \binom{n-k}{\left\lfloor\frac{n-k}{Q}\right\rfloor} < \left( \frac{(n-k)e}{\frac{n-k}{Q}}\right)^{\frac{n-k}{Q}} < e^{\frac{n-k}{Q} \ln\frac{(n-k)e}{(n-k)/Q}} = e^{\frac{n-k}Q \ln(Qe)}.
	\end{aligned} 
	\end{equation} 
	
	Finally, observe: 
	\begin{equation}\label{eq:q2sq-est}
	q^2 s^q = e^{2\ln q + q\ln s} < e^{2q + q\ln s} < e^{\frac{n-k}{Q}(2+\ln s)}.
	\end{equation}
	Given sufficiently small $\epsilon_0 > 0$, one can choose $Q$ large enough so that $$\frac 1 Q \left(2+\ln s + \ln(2e) + \ln(Qe)\right) < \epsilon_0.$$ Applying this to the estimates \eqref{eq:(k-1)(q-1)-est}, \eqref{eq:(n-k-1)(q)-est}, and \eqref{eq:q2sq-est}, we obtain:
	\[
	q^2s^q \binom{k-1}{q-1}\binom{n-k-1}{q} e^{h(n-k)} \leq e^{(n-k)(h+\epsilon_0)}
	\]
	and therefore, from \eqref{eq:combinatorial-approx},
	\[
	\#\Ss_{k,n,q} \leq \frac{C_8}{q^2} e^{(h+\epsilon_0)(n-k)}.
	\]
\end{proof}

\begin{lemma}\label{lem:PSS8.2}
	There is an $\epsilon_0 > 0$ such that for any $\Lambda_i^s \in \Ss_{k,n,q}$, 
	\[
	\mu_1\left(\Lambda_i^s\right) \leq C_9 k^{-\gamma'} e^{(-\log\lambda + \epsilon_0)(n-k)},
	\]
	where $C_9>0$ is a constant and $\gamma'$ is given in \eqref{gamma-defs}. 
\end{lemma}

\begin{proof}
	Let $x \in \Lambda_i^s$, and let $\gamma_i^s(x) = \gamma_i^s \subset \Lambda_i^s$ be the connected component of the stable manifold of $x$ intersected with $\Lambda_i^s$ that contains $x$. By \eqref{TLB 1}, we have $\mu_1(\Lambda_i^s) = K_1 L(g^n(\gamma_i^s(x)))$. Note further that the length of the backwards iterates of $\gamma_i^s$ lying outside of $\Us_{0}$ are stretched by the expansion factor $\lambda$ of the pseudo-Anosov homeomorphism $f$. Additionally, the time spent in $\Us_{0} \setminus \Us_{1}$ is uniformly bounded by Lemma \ref{annular bound in M}, and therefore whenever the orbit of $\gamma_i^s$ enters $\Us_{0}$, we can use Lemmas \ref{annular bound in M} and \ref{lem:length ratio is polynomial} to give an upper bound for its length. So, letting $\{k_j \geq 0\}_{j=1,\ldots, q}$ and $\{l_j \geq 0\}_{j=0,\ldots, q}$ be such that the orbit of $\gamma_i^s$ spends $k_j$ times consecutively inside $\Us_{0}$ and $l_j$ times outside $\Us_{0}$, we obtain: 
	\begin{equation}\label{eq:lem8.2-1}
	\begin{aligned}
	\mu_1(\Lambda_i^s) &= K_1 L(g^n(\gamma_i^s)) \\
	&\leq K_1 \lambda^{-l_q}L(g^{n-l_q}(\gamma_i^s)) \\
	&\leq K_1 C_6 k_q^{-\gamma'} \lambda^{-l_q} L(g^{n-l_q - k_q}) \\
	&\vdots \\
	&\leq K_1 C_6^q \lambda^{-(l_q + \cdots + l_0)} \left(k_q k_{q-1} \cdots k_1\right)^{-\gamma'} L(\gamma_i^s). 
	\end{aligned}
	\end{equation} 
	Since $\gamma_i^s(x)$ is a full-length stable curve in $R$, its length is independent of $n=\tau(\Lambda_i^s)$. So we may take $K_1 L(\gamma_i^s) \leq K_1'$ for some $K_1' > 0$. Furthermore, if $\rho_0$ is made sufficiently small, the time $k_i$ that the orbit stays in $\Us_{0}$ may be made to be $\geq 2$. Therefore, 
	\begin{equation}\label{eq:lem8.2-2}
	k_1 k_2 \cdots k_q \geq 2^{q-1} \max_{1 \leq i \leq q} k_i \geq q\max_{1 \leq i \leq q} k_i \geq \sum_{i=1}^q k_i = k.
	\end{equation}
	Finally, $C_6^q = e^{q\ln C_6} < e^{\frac{n-k}{Q} \ln C_6} < e^{\epsilon_0(n-k)}$ for sufficiently small $\epsilon_0 > 0$ and sufficiently large $Q \geq 1$. Therefore, applying this estimate and \eqref{eq:lem8.2-2} to \eqref{eq:lem8.2-1}, we obtain:
	\[
	\mu_1(\Lambda_i^s) < K_1' e^{\epsilon_0(n-k)} \lambda^{-(n-k)} k^{-\gamma'} < C_9 k^{-\gamma'} e^{(-\log\lambda + \epsilon_0)(n-k)}.
	\]
\end{proof}

\begin{lemma}\label{tail upper bound}
	There exists a constant $C_{10} > 0$ such that 
	\[
	\mu_1(\{x \in \Lambda : \tau(x) > n\}) < C_{10} n^{-(\gamma'-1)},
	\]
	where $\gamma' > 0$ is defined in \eqref{gamma-defs}.
\end{lemma} 

\begin{proof}
	Observe that: 
	\[
	\mu_1(\{x \in \Lambda : \tau(x) = n\}) \leq \sum_{k=1}^n \sum_{q=1}^k \left( \max_{\Lambda_i^s \in \Ss_{k,n,q}}\mu_1\left( \Lambda_i^s\right)\right) \#\Ss_{k,n,q}.
	\]
	It follows from Lemmas \ref{lem:PSS8.1} and \ref{lem:PSS8.2} that: 
	\begin{align*}
	\mu_1(\{x \in \Lambda : \tau(x) = n\}) &\leq \sum_{k=1}^{n} \sum_{q=1}^k C_8 C_9 \frac 1{q^2} k^{-\gamma'} e^{(-\log\lambda+\epsilon_0)(n-k)}e^{(h+\epsilon_0)(n-k)} \\
	&< C_8 C_9 \frac{\pi^2}{6} e^{-\delta n} \sum_{k=1}^n k^{-\gamma'} e^{\delta k}
	\end{align*}
	where $\delta = \log\lambda - h - 2\epsilon_0 > 0$ if $\epsilon_0 > 0$ is sufficiently small.
	
	To estimate $\sum_{k=1}^{n} k^{-\gamma'} e^{\delta k}$, let $u_k = k^{-\gamma'} e^{\delta k}$, and note that: 
	\[
	u_{k+1} - u_k = e^{\delta k} k^{-\gamma'} \left( e^\delta \left( \frac{k}{k+1}\right)^{\gamma'}-1\right) \sim e^{\delta k}k^{-\gamma'},
	\]
	where $a_k \sim b_k$ means $\lim_{k \to \infty} \frac{a_k}{b_k}$ exists and is $>0$ for positive sequences $a_k$ and $b_k$. It follows that:
	\[
	\sum_{k=1}^n u_k \sim \sum_{k=1}^n u_{k+1} - u_k = u_{n+1} - u_1 \sim e^{\delta n} n^{-\gamma'},
	\]
	where the first asymptotic comparison comes from the Stolz-Ces\`areo theorem, since $u_k > 0$ for all $k$ and the series $\sum_{k=1}^\infty u_k$ diverges. Therefore, there is a $C_{9}' > 0$ for which 
	\[
	\mu_1(\{x \in \Lambda : \tau(x) = n\}) \leq C_8 C_9 \frac{\pi^2}{6} e^{-\delta n} \sum_{k=1}^n u_k \leq C_{9}' n^{-\gamma'}. 
	\]
	It follows that there is a $C_{10} > 0$ independent of $n$ for which: 
	\[
	\mu_1(\{x \in \Lambda : \tau(x) > n\}) = \sum_{k>n} \mu_1(\{x \in \Lambda : \tau(x) = k\}) < C_{10} n^{-(\gamma' - 1)}.
	\]
	This concludes the proof of the Lemma and the upper bound on the tail of the return time.
\end{proof}

\section{Proof of Theorem \ref{Main 1}}

We now prove the main result. Statements (1) and (2) of Theorem \ref{Main 1} are shown in Propositions \ref{prop:SRB} and \ref{prop:topological-properties}. To show $g : M \to M$ is Bernoulli with respect to $\mu_1$, note the pseudo-Anosov homeomorphism $f : M \to M$ also has an invariant area measure $m_1$ by \cite{FS79}, which is absolutely continuous with respect to $\mu_1$. The conjugating homeomorphism $h : M \to M$ for which $f = h \circ g \circ h^{-1}$ is $C^1$ away from the singularities of $f$; in particular, $h$ transfers (un)stable manifolds of $g$ to (un)stable manifolds of $f$. It follows that the measure $h_* \mu_1$ is an $f$-invariant SRB measure. Since $f$ is topologically transitive, its SRB measure is unique by \cite{RH2TU}, so $h_* \mu_1 = m_1$. Since $(M, f, m_1)$ is Bernoulli by \cite{FS79}, and $f$ and $g$ are measure-theoretically isoomorphic, $g$ is also Bernoulli. This proves Statement (3). Statement (7) follows from Proposition \ref{prop:PAD-thermodynamics}, Statement (1), when $t=0$. 

In the remainder of the section, we will show that $g$ has polynomial upper and lower bounds on the decay of correlations (Statement (4)), that $g$ satisfies the CLT (Statement (5)), and that $g$ has polynomial large deviations (Statement (6)).

\subsection{Decay of correlations}

By Proposition \ref{exists-Q}, the pseudo-Anosov smooth model $g : M \to M$ is a Young diffeomorphism with base $\Lambda$, $s$-sets $\Lambda_i^s$, and inducing times $\tau = \{\tau_i\}$, $\tau_i = \tau(\Lambda_i^s)$. The associated \emph{Young tower} is the space
\[
\hat Y = \{(x,k) \in \Lambda \times \N_0 : 0 \leq k < \tau(x)\}
\]
and the associated map $\hat g : \hat Y \to \hat Y$ is given by 
\[
g(x,k) = \begin{cases}
(x,k+1) & \textrm{if } 0 \leq k < \tau(x)-1, \\
(g(x),0) & \textrm{if } k = \tau(x)-1.
\end{cases}
\]
Define the subsets $\hat M_k \subset \hat Y$ by
\[
\hat M_k = \{(x,\ell) \in \hat Y : 0 \leq \ell \leq \min\{k,\tau(x)\}\}.
\]
Note $M_k$ is the set of the first $k$ levels of the Young tower. Finally define the projection $\pi : \hat Y \to M$ by $\pi(x,\ell) = g^\ell(x)$, and define
\[
Y = \pi(\hat Y), \quad M_k =  \pi(\hat M_k).
\]
Note that the sets $M_k$ are nested and exhaust $Y$. 

Proving that $g : M \to M$ admits upper and lower bounds on polynomial decay of correlations requires the following result, which follows from Theorem 2.3 in \cite{SZ-inducing} and Theorem 7.1 in \cite{PSZTowers} and its proof: 

\begin{proposition}\label{prop:dark-tower}
	Assume that: 
	\begin{itemize}
		\item the greatest common divisor of the inducing times $\tau_i = \tau(\Lambda_i^s)$ is 1; 
		\item there is a constant $C>0$ for which for all $\Lambda_i^s \subset \Lambda$, all $x, y \in \Lambda_i^s$, and all $0 \leq k \leq \tau_i$, 
		\[
		d(f^k(x), f^k(y)) \leq C\max \left\{d(x,y), d(f^{\tau_i}(x), f^{\tau_i}(y))\right\};
		\]
		\item there are constants $\theta > 1$ and $B > 0$ such that 
		\[
		m(\tau > n) \leq Bn^{-\theta}.
		\]
	\end{itemize}
	Then the following statements hold: 
	\begin{enumerate}[label=\emph{(\alph*)}]
		\item There is a constant $B_1 > 0$ such that $\mathrm{Cor}_n(h_1, h_2) \leq B_1 n^{1-\theta}$ for any $h_1, h_2 \in C^\eta(M)$.
		\item For any $h_1, h_2 \in C^\eta(M)$ supported in $M_k$ for some $k > 0$, we have: 
		\begin{equation}\label{eq:PSS-37}
		\mathrm{Cor}_n(h_1, h_2) = \sum_{k=n+1}^\infty m(\tau(x) > k) \int h_1 \,dm \int h_2 \,dm + O(R_\theta(n)),
		\end{equation}
		where:
		\[
		R_\theta(n) = \begin{cases}
			n^{-\theta} & \textrm{if } \theta > 2, \\
			n^{-2}\log n & \textrm{if } \theta = 2, \\
			n^{-2(\theta-1)} & \textrm{if } 1 < \theta < 2.
		\end{cases} 
		\]
		Moreover, if $\int h_1 \,dm \int h_2\,dm = 0$, then $\mathrm{Cor}_n(h_1, h_2) = O(n^{-\theta})$. 
	\end{enumerate}
\end{proposition}

\begin{remark}
	The consequences in Proposition \ref{prop:dark-tower} are the same as in Theorem 2.3 in \cite{SZ-inducing}. There, the authors prove these results in higher generality for equilibrium states for a given potential; however, in addition to the above assumptions, the potential is assumed to satisfy certain conditions ((P1) - (P4) in \cite{SZ-inducing} and \cite{PSZTowers}). In the proof of Theorem 7.1 in \cite{PSZTowers}, it is shown that the geometric potential $\phi(x) = -\log\left| dg|_{E^u(x)}\right|$ of a Young diffeomorphism $g$ satisfies conditions (P1) - (P4) in \cite{SZ-inducing}. Since the pseudo-Anosov smooth model $g : M \to M$ is a Young diffeomorphism, it remains only to verify the assumptions in Proposition \ref{prop:dark-tower} to apply the result to the pseudo-Anosov smooth model $g$. 
\end{remark}

\begin{proof}[Proof of Theorem \ref{Main 1}, (4)] We begin by proving the upper bound (statement (a) of Theorem \ref{Main 1}, (4)). By Proposition \ref{prop:dark-tower}, the claim is immediate once we verify the three conditions of the proposition. 
	
	First, recall that $g$ is topologically conjugate to the pseudo-Anosov homeomorphism $f$. Since $f$ is Bernoulli \cite{FS79}, every power of $f$ is ergodic. If $\tilde\Lambda = \union_{i \geq 1} \tilde\Lambda_i^s$ is the base of the Young structure for $f$ and the inducing times are $\tilde\tau : \tilde\Lambda \to \N_0$ (see Section \ref{subsec:Realizing $g$ as a Young diffeomorphism}), then $\gcd(\tilde\tau_i) = 1$ (where $\tilde\tau_i = \tilde\tau(\tilde\Lambda_i^s)$), and so $\gcd(\tau_i) = 1$. 

The second assumption of Proposition \ref{prop:dark-tower} follows from the fact that $g : M \to M$ has a Young structure, and the images $g^k(\Lambda_i^s)$, $1 \leq k \leq \tau_i - 1$, have diameter less than the diameter of the Markov partition. 

Finally, the third assumption holds because by Lemmas \ref{tail lower bound} and \ref{tail upper bound}, we have
\begin{equation}\label{eq:38}
\frac{C_8}{n^{\gamma-1}} < \mu_1(\{x \in \Lambda : \tau(x) > n\}) < \frac{C_{10}}{n^{\gamma' - 1}}
\end{equation}
where $\gamma, \gamma'$ are defined in \eqref{gamma-defs}. Note for $0 < \alpha < \frac 1 4$ and $0 < \mu < \frac 1 2$ we have that $\gamma > \gamma' > 2$, so the third assumption holds. Statement (a) of Proposition \ref{prop:dark-tower} gives the upper bound on the decay of correlations (statement (4)(a) of Theorem \ref{Main 1}), using $\gamma_1 = \gamma' - 2 > 0$. 

To prove the lower bound, by Statement (b) of Proposition \ref{prop:dark-tower} with $\theta = \gamma' - 1 > 1$, we have that for all $h_1, h_2 \in C^\eta(M)$ supported in $M_k$, for some $k > 0$:
\begin{equation}\label{eq:39}
\mathrm{Cor}_n(h_1, h_2) = \sum_{k=n+1}^\infty \mu_1(\{x : \tau(x) > k\}) \int_M h_1\,d\mu_1 \int_M h_2 \, d\mu_1 + O(R_{\gamma'-1}(n)),
\end{equation}
where we recall from the definition of $R_{\theta}$ in Proposition \ref{prop:dark-tower} that
\[
R_{\gamma'-1}(n) = \begin{cases}
	n^{-\gamma'+1} & \textrm{if } \gamma' > 3, \\
	n^{-2} \log n & \textrm{if } \gamma' = 3, \\
	n^{-2(\gamma - 2)} & \textrm{if } 2 < \gamma' < 3.
\end{cases}
\]
We consider separately the three cases $\gamma' > 3$, $2 < \gamma' < 3$, and $\gamma' = 3$.

If $\gamma' > 3$ (which is guaranteed if $\alpha < \frac 1 6$), then by the assumption in Statement (4)(b) of Theorem \ref{Main 1}, we may apply \eqref{eq:38} and \eqref{eq:39} above and obtain:
\begin{equation}\label{eq:40}
\mathrm{Cor}_n(h_1, h_2) > \sum_{k=n+1}^\infty K_1' n^{-(\gamma-1)} + K_2 n^{-(\gamma' - 1)} > K_1 n^{-(\gamma-2)} + K_2 n^{-(\gamma'-1)}
\end{equation}
for constants $K_1$ and $K_2$ (depending on $h_1, h_2$). By the definitions of $\gamma$ and $\gamma'$ in Equation \eqref{gamma-defs}, after choosing $0 < \mu < \frac 1 2$, we can show $\gamma - 2 < \gamma' - 1$ for all $0 < \alpha < \frac 1 6$. So there is a $C > 0$ for which
\[
\mathrm{Cor}_n(h_1, h_2) > \frac{C}{n^{\gamma-2}}.
\]

Now consider the case where $\frac 1 6 < \alpha < \frac 1 4$. In this situation, $\gamma' > 2$, but depending on the value of $\mu$, we may have either $\gamma' > 3$ or $\gamma' < 3$. We assume the latter; otherwise we're back in the first case. With this assumption, similar to \eqref{eq:40}, we can use \eqref{eq:38} and \eqref{eq:39} to show: 
\[
\mathrm{Cor}_n(h_1, h_2) > K_1 n^{-(\gamma-2)} - K_3 n^{-2(\gamma'-2)},
\]
for some constants $K_1$ and $K_3$ depending on $h_1$ and $h_2$. As before, choosing $0 < \mu < \frac 1 2$, one can show $\gamma - 2 < 2\gamma' - 4$ for all $0 < \alpha < \frac 1 4$. This gives us the estimate 
\[
\mathrm{Cor}_n(h_1, h_2) > Cn^{-(\gamma-2)}
\]
for some $C>0$ and all $0 < \alpha < \frac 1 4$. In particular, we have $\gamma_2 = \gamma-2 > 0$. This gives us both (a) and (b) of Statement (4) of Theorem \ref{Main 1}.
\end{proof}


\subsection{The Central Limit Theorem}

\begin{proof}[Proof of Theorem \ref{Main 1}, (5)] By Statement (4)(a) of Theorem \ref{Main 1}, if $h \in C^\eta$ satisfies $\int h \, d\mu_1 = 0$, then $\mathrm{Cor}_n(h,h) = O(n^{-(\gamma'-1)}$. Therefore the correlation function is summable for $\gamma' > 2$. It follows from Theorem 1.1 of \cite{Liv95} that the system $(M, g, \mu_1)$ has the central limit theorem with respect to H\"older potentials. 
\end{proof}

\subsection{Large Deviations}

\begin{proof}[Proof of Theorem \ref{Main 1}, (6)] 
	Because the map $g : M \to M$ is modeled  by a Young tower, the upper bound in \eqref{eq:38} allows us to use Theorem 4.2 in \cite{MN08} to show that for $0 < \alpha < \frac 1 4$ (so that $\gamma' > 2$), and for all sufficiently small $a > 0$ and all H\"older $h : M \to \R$, there is a constant $C = C_{h,a}$ depending continuously on $h$ (in the $C^\eta$ topology) such that for all $\epsilon > 0$ and all sufficiently large $n \geq 0$, 
	\[
	\mu_1\left(\left\{ \left| \frac 1 n \sum_{i=0}^{n-1} h(g^i(x)) - \int h \,d\mu_1\right| > \epsilon \right\}\right) < C_{h,a} \epsilon^{-2(\gamma'-2-a)} n^{-(\gamma' - 2- a)}.
	\]
	This proves the first part of Statment (6) of Theorem \ref{Main 1}. To obtain a lower bound on the large deviations, we will use Theorem 4.3 in \cite{MN08}. The one condition of this theorem that needs to be checked is that $\mu_1(\overline M_k) < 1$ for some $k \geq 0$, where we recall $M_k = \pi(\hat M_k)$ and 
	\[
	\hat M_k = \{(x,\ell) \in \hat Y : 0 \leq \ell \leq \min \{k,\tau(x)\} \}
	\]
	and $\pi : \hat Y \to M$ is the projection $\pi(x,k) = f^{k}(x)$ for $x \in \Lambda$, $0 \leq k \leq \tau(x)-1$. 
	
	Given $k \geq 0$, choose a partition element $\Lambda_i^s$ in the base $\Lambda$ of the Young tower with $\tau(\Lambda_i^s) \leq k$. Identifying $\Lambda_i^s$ with a subset of the 0-level of the tower $\hat Y$, we see $\Lambda_i^s \subset \hat Y \setminus \hat M_k$, and we see that $\hat\mu_1(\Lambda_i^s) > 0$, where $\hat\mu_1$ is the lifted measure of $\mu_1$ to the tower $\hat Y$. It follows that $\hat\mu_1(\hat M_k) < 1$, and since the projection $\pi : \hat Y \to M$ is measure-preserving, $\mu_1(M_k) < 1$. So, by Theorem 4.3 in \cite{MN08}, for small $\epsilon > 0$, an open and dense subset of H\"older observables $h$, and infinitely many $n$, we obtain the lower bound 
	\[
	n^{-(\gamma' - 2 + a)} < \mu_1\left(\left\{ \left| \frac 1 n \sum_{i=0}^{n-1} h(g^i(x)) - \int h \,d\mu_1\right| > \epsilon \right\}\right)
	\] 
\end{proof}

\end{document}